\documentclass[a4paper,12pt]{article}

\usepackage{amsmath,amsfonts,amssymb,amsthm}
\usepackage{times}
\usepackage{graphicx}
\usepackage{xcolor}
\usepackage{hyperref}
\hypersetup{
    colorlinks,
    linkcolor={red!60!black},
    citecolor={blue!80!black},
    urlcolor={blue!80!black}
}
\newcommand{\bR}{\mathbb{R}}
\newcommand{\ind}{{\bf 1}}
\newcommand{\bZ}{\mathbb{Z}}

\newcommand{\bN}{\mathbb{N}}
\newcommand{\bE}{\mathbb{E}}
\newcommand{\bP}{\mathbb{P}}

\newcommand{\cF}{{\cal F}}

\newcommand{\cB}{{\cal B}}

\newcommand{\sX}{{\sf X}}

\newtheorem{theorem}{Theorem}

\newtheorem{definition}[theorem]{Definition}

\newtheorem{lemma}[theorem]{Lemma}

\newtheorem{proposition}[theorem]{Proposition}

\newtheorem{assumption}{Assumption}[section]
\theoremstyle{remark}
\newtheorem{remark}{Remark}

\begin{document}

\title{\bf \Large Gated Infinite Server Queues in Light Traffic}
\author{Dimitra Pinotsi and Michael A. Zazanis\thanks{Corresponding author.
		{\tt zazanis@aueb.gr}}\\ Department of Statistics\\
Athens University of Economics and Business\\ Athens, Greece 10434}
\date{}

\maketitle

\begin{abstract}
	We consider $M/GI/\infty$ queues with gated service and obtain results on the
	distribution of the service stage length and the number of customers served
	in a stage when the system is stationary. Under an exponential moment condition
	for the service time distribution, together with a light traffic condition, the service 
	stage length density is expressed as an infinite series of terms, involving the solution 
	of an infinite linear system. The existence and uniqueness of the solution of this 
	system is established. It is also shown that it can be obtained as the limit of finite 
	sections (truncations) of this infinite system. For service time distributions that have 
	exponential moments of all orders the restriction to light traffic is not needed. 
	Analogous results are established for a related $GI/M/\infty$ gated system and 
	numerical results illustrating the method are given.\\
	
\end{abstract}

\section{Introduction: The gated, exhaustive, parallel service system}

We consider an infinite server gated queue with Poisson arrivals. 
Customer service times are independent, identically distributed and service is in stages,
controlled by a gate. At the beginning of a stage the gate opens and all waiting
customers are admitted for service. 
The gate closes behind them immediately and the customers admitted are served 
in parallel. When all customers have completed their service the current service 
stage ends, the gate opens again admitting to service all arrivals since it last opened,
and a new service stage begins. If no customers
arrived during a stage, the system remains idle until a customer arrives, 
the gate opens, and a service stage with a single customer begins. Depending 
on the application, this policy may be far from optimal in terms of system performance, 
but the focus of this paper is to obtain analytic results for the gated parallel 
service system in a simple framework containing only essential characteristics, rather 
than to investigate the performance of various operational 
policies for it. This model was first considered by Browne et al.\ \cite{Browne1}. 

Specifically, customers arrive at an infinite server system according to a 
Poisson stream with rate $\lambda$. Their service requirements are i.i.d.\ 
random variables $\{\sigma_i\}_{i=1,2,\ldots},$ 
with common distribution $G(x) =\bP(\sigma_1\leq  x)$. 
Denote the number of customers served in the $n$th stage by $K_n$
and its duration by $Y_n$. Then
$\{K_n\}_{n\in \mathbb{N}}$ is a Markov Chain with state space 
$\bZ^+:=\{1,2,3,\ldots,\}$ and $\{Y_n\}_{n\in \mathbb{N}}$ is 
a Discrete Time Continuous State Markov Chain with  
state space $\mathbb{R}^+$. (This will be shown in 
section \ref{sec:GQ}. See also \cite{Browne1} and \cite{PM1}.) 
An analysis of these Markov Chains is given in \cite{Browne1}, 
\cite{Browne2}, together with approximations to their stationary 
distribution in a number of cases. 
The stability condition for this system is not obvious if the service time distribution
does not have bounded support. In \cite{Browne1} it was shown that the system
is stable provided that the {\em second moment} of the service time distribution,
$\int_0^\infty x^2 dG(x)$, is finite. This result was improved upon in \cite{PM1} 
where it was established by means of a Foster--Lyapunov 
drift criterion that the finiteness of the {\em first moment} $\int_0^\infty x dG(x)$  
is both necessary and sufficient for the positive recurrence of a suitably defined
underlying Markov Chain, which in turn is used to establish the stability 
of the system. This stability result is also contained in the more general 
framework of \cite{Aldousetal}. 

In Browne et al. \cite{Browne2} the infinite server gated system with Poisson arrivals and 
uniformly distributed service times is considered and an integral equation for the
probability generating function of the number of customers served in a stage
in stationarity is derived. Part of the analysis in \cite{Browne2} involves transforming 
this integral equation into an infinite system of linear equations which may be written
in the form $x=c+Ax$. The analysis is then carried out in $l^2$ and $A$
is shown to be a Hilbert-Schmidt operator with $\Vert A \Vert_2 <1$. Our approach has 
some common elements with the approach used in \cite{Browne2}. However the scope of our
results is more general since we do not assume that the service times are uniform. In our
case the resulting infinite system is analyzed in $l^1$ and $A$ is a bounded linear 
operator whose norm is not necessarily less than 1. We discuss these matters further in 
section \ref{sec:number_k}. 

A heavy traffic analysis of this system has been carried out 
in Tan and Knessl \cite{Knessl1}. Lebedev examines asymptotics for the number of 
customers in a stage and the stage length in heavy traffic when the service time 
distribution is in the domain of attraction of the Gumbel law in \cite{Lebedev1} when 
the service time distribution follows a power law and is in the Frech\'et domain of 
attraction in \cite{Lebedev2}. Other queueing systems with gated mechanisms have 
been considered in \cite{AH}, \cite{JS}, and \cite{RS}.

We focus on the service stage length process $\{Y_n\}$ of the gated 
$M/GI/\infty$ queue. We assume that the service time 
distribution has a density, $g$, that satisfies the exponential moment condition 
$\int_0^\infty e^{R x} g(x)dx < \infty$ {\em for some $R>0$} and that the Poisson
arrival rate $\lambda$ satisfies the {\em light traffic condition $\lambda <R$}.
Under this assumption we will show that the service stage lengths also satisfy an 
exponential moment condition of order $R$ and we will derive an infinite linear system
satisfied by the moments of the service stage length which will be shown to have 
a unique solution. 
We also establish that this solution is the limit of the solutions the sequence 
of the $N\times N$ truncated systems (including the first $N$ 
equations and the first $N$ variables) known as {\em sections} of the infinite
system. Numerical results illustrating the 
quality of these successive approximations are given. 
We point out that this technique applies equally well to the case of service 
times with {\em superexponential tails} i.e.\ to service times that satisfy 
$\int_0^\infty e^{R x}g(x)dx<\infty$ {\em for all $R>0$}. 
In this case the method is valid for any value of $\lambda$. In particular, the method
holds when the service time density has bounded support. The case examined in
\cite{Browne2} where service times are uniformly distributed in $[0,1]$ is of course
a special case of this.
 
A second related system, the {\em synchronized} gated $GI/M/\infty$ system,
is analyzed in section 3. Here the process of the number of customers 
served in each stage, $\{K_n\}$ is analyzed. This is an irreducible,
aperiodic Markov chain with state space $\bZ^+$. It is shown by means 
of a Foster-Lyapunov argument that it is positive recurrent and that, 
under a light traffic condition its stationary distribution has geometric moments
of order greater than 2. In this  light traffic case it is
shown that the descending factorial moments of this distribution
satisfy an infinite linear system which has a unique bounded
solution. This may be again obtained as the limit of the
solutions of truncated $N\times N$ systems. The stationary distribution
is then obtained again in terms of the factorial moments. 
 
\section{The gated $M/GI/\infty$ system and the service stage length density} \label{sec:GQ}

Customers arrive to the system according to a Poisson process with rate $\lambda$.
We denote the points of this process on $(0,\infty)$ by $\{S_m\}$, $m=1,2,\ldots$
and the corresponding counting measure by $N$. Thus, for any interval 
$I \subset (0,\infty)$ $N(I) = \sum_{m \in \bN} \ind(S_m \in I)$.  $\sigma_m$ 
denotes the $m$th customer's service requirement and $\{\sigma_m\}$
is an i.i.d. sequence of positive random variables, independent of the $\{S_m\}$, 
with common distribution function $G(x):= \bP(\sigma_1 \leq x)$, $x \geq 0$. 
The sample path of the gated $M/GI/\infty$ system consists of consecutive 
{\em service cycles.}

Suppose the system begins its operation at $T_0=0$ with the gate opening and 
admitting $K_1$ customers to the service station. The gate immediately closes again
and any customers that arrive subsequently stay in the waiting area. The service station 
has infinite capacity and serves customers in parallel so
the time to serve the customers admitted is  $Y_1 := \max(\sigma_1,\ldots,\sigma_{K_1})$. 
We  call this the {\em service stage}. The number of arrivals during 
this time is $N(0,Y_1]$. If $N(0,Y_1]>0$ then the service cycle is complete when its 
service stage ends and has length $T_1=T_0 + Y_1$. In that case the 
gate opens, and the customers that arrived during the first stage are admitted to the 
service station. These will be served during the second service stage which commences 
immediately at time $T_1:=T_0+Y_1$. If however $N(0,Y_1]=0$ then the first service
cycle is extended by a {\em waiting phase} which lasts until the first 
arrival time $S_1$. Therefore, the end of the first service cycle and the beginning of the
second occurs at $T_1:= \max\{T_0+Y_1, S_1\}$. (On the event 
$S_1 >T_0+Y_1$, the first cycle includes a waiting phase which is exponentially distributed 
with rate $\lambda$.)
The second service cycle begins at $T_1$ when the gate opens momentarily, 
and the customers that arrived during the first cycle are admitted for service. 
There are $K_2= N(0,T_1]$ such customers. Thus each stage 
serves at least one customer. (If a service {\em cycle} includes a waiting 
phase, then the next service stage serves a single customer.)

The $n$th service cycle begins at $T_{n-1}$ and its duration 
is $T_{n}-T_{n-1}$. The length of the $n$th service stage is  $Y_n$. The $K_{n+1}$ 
customers who arrived during the $n$th service cycle are admitted for service, initiating 
the $(n+1)$th service stage. (If no customers arrived during the $n$th stage and
the $n$th cycle includes a waiting phase then necessarily 
$K_{n+1}=1$.) If we define the index sequence 
$\{L_n\}$ as $L_0=0$, $L_n=L_{n-1}+K_n$, $n=1,2,\ldots$ then we can describe
the operation of this systems as follows:
\begin{eqnarray*}
&&K_1 \geq 1 \mbox{ given}, \; L_0=0,\; T_0=0, \;\; \mbox{ and for $n=1,2,\ldots,$ } \\
&&L_n=L_{n-1}+K_n, \;\;\; 
Y_n= \max\{\sigma_{L_{n-1}+1}, \sigma_{L_{n-1}+2},\ldots, \sigma_{L_n}\}, \\
&& \displaystyle \;\;\; \mbox{ if $N(T_{n-1}, T_{n-1}+Y_n] >0\;\;$ then }\;
T_{n}  = T_{n-1} +Y_n \;\; \mbox{ and } \; K_{n+1} = N(T_{n-1}, T_{n-1}+Y_n] \\
&& \displaystyle \;\;\; \mbox{ if $N(T_{n-1}, T_{n-1}+Y_n] =0\;\;$ then }\; 
T_n = \inf\{S_m: S_m>T_{n-1}\}  \; \mbox{ and } K_{n+1} = 1 . 
\end{eqnarray*}
The sample paths of the system consist thus of semi-regenerative cycles 
(see \cite[p. 211]{Asmussen}) corresponding to the intervals $[T_n,T_{n+1})$. 
The process $\{(T_n,K_n);n\in \bN\}$ is a Markov-Renewal
process with transition kernel
\begin{eqnarray} \label{MRP}
 \bP(T_{n+1}-T_n \leq t ; K_{n+1}=j \,|\, K_n=i)  &&  \\  \nonumber
 &&  \hspace{-1.5in}
 \;=\; \left\{
 \begin{array}{lcr}  \displaystyle 
 \int_0^t i G(u)^{i-1} \frac{(\lambda u)^j}{j!}
 e^{-\lambda u} dG(u) & & j \geq 2, \\  &&  \\   \displaystyle
  \int_0^t i G(u)^{i-1} \lambda u  e^{-\lambda u}dG(u) + 
 \int_0^t G(u)^i \lambda e^{-\lambda u} du & & j =1. \end{array}
 \right. 
\end{eqnarray}
(In the case $j=1$, the second term of the transition kernel corresponds
to the event where no customers arrive during the active part of the 
service stage and the gate remains open for a period of time until
the next Poisson arrival.) 
$\{K_n\}$ is a Markov chain with state space $\bZ^+$ and
transition probability $P_{ij}$ obtained from the kernel (\ref{MRP}) if we let $t=+\infty$.
Clearly it is irreducible and aperiodic since  
$\bP(K_{n+1}=j | K_n =i) >0$ for any $i, j \in \bZ^+$.



We assume that the customer service time distribution $G(x)=\bP( \sigma_1 \leq x)$, 
is absolutely continuous with respect to the Lebesgue measure with
density $g(x)= G'(x)$. We also set $\overline{G}(x):=1-G(x)$. Let 
$\sX :=[0,\infty)$ and $\cB$ the family of Borel sets of $[0,\infty)$.
\begin{proposition} \label{prop:Q} The sequence of consecutive service stage lengths 
$\{Y_n\}$ is a Markov chain with state space ${\sf X}:=[0,\infty)$ and transition kernel 
$Q(x,A)= \int_A q(x,y)dy$ for $x \in {\sf X}$, $A \in \mathcal{B}$, which
is absolutely continuous with respect to the Lebesgue measure for all $x$ with
density 
\begin{equation}  \label{transition-density}
q(x,y)= \left( \lambda xe^{\mathbf{-}\lambda x\overline{G}(y)}
+e^{-\lambda x}\right)g(y), \hspace{0.2in} y \geq 0.
\end{equation}
\end{proposition}
\begin{proof}
Define the filtration $\{\mathcal{F}_n\}_{n\in \bN}$ as  
$\mathcal{F}_n:=\sigma-\{K_{n+1},Y_n,K_n,\ldots,K_2,Y_1,K_1\}$
and similarly  $\{\mathcal{F}^Y_n\}_{n\in \bN}$ as
$\mathcal{F}_n^Y:=\sigma-\{Y_n,Y_{n-1},\ldots,Y_1\}$. Note that 
$\mathcal{F}_n^Y \subset \mathcal{F}_n$. 
Given $K_{n+1}$ the duration of stage $n+1$, $Y_{n+1}$, is the maximum of 
$K_{n+1}$ independent service times: 
\begin{equation}  \label{YK}
\bP(Y_{n+1} \leq y \mid \mathcal{F}_n) = 
\bP(Y_{n+1} \leq y \mid  K_{n+1}) = G(y)^{K_{n+1}}.
\end{equation}
In turn, given $\cF_n^Y$, the probability that stage $n+1$ serves $k$
customers is
\begin{equation} \label{KY}
\bP(K_{n+1} =k\mid \cF_n^Y) \;=\; \bP(K_{n+1} =k\mid Y_n)=\left\{
\begin{array}{lll}
\frac{1}{k!}(\lambda Y_n)^k e^{-\lambda Y_n}, && \mbox{for } k=2,3,\ldots,\\
&& \\
(\lambda Y_n +1) e^{-\lambda Y_n} && \mbox{for }k=1.
\end{array} \right.
\end{equation}
The process $\left\{ Y_n \right\}_{n \in \bN}$ constitutes a 
Discrete Time Markov Chain with state space $\bR^+$. Indeed, for any $y>0$, 
using the tower property of conditional expectations, (\ref{YK}), and (\ref{KY}) 
we have
\begin{eqnarray} \nonumber
\bP\Big(Y_{n+1}\leq y \mid \cF^Y_n\Big) &=& 
\bE \Big[ \bP( Y_{n+1} \leq y \mid  \cF_n) \mid \cF_n^Y\Big]  \;=\;
\bE \Big[ \bP( Y_{n+1} \leq y \mid  K_{n+1}) \mid \cF_n^Y\Big]  \\
&=& \bP\big(Y_{n+1} \leq y \mid Y_n\big).
\end{eqnarray}
This establishes the markovian property for the process $\{Y_n\}$.
(From (\ref{KY}) we also see that $K_{n+1}$ is conditionally 
independent of $Y_{n-1},Y_{n-2},\ldots, Y_1$ given $Y_n$.) 
Its transition kernel, denoted by $Q(x,dy) := \bP(Y_{n+1}\in dy\mid Y_n = x)$, 
can be obtained using the relationship between $K_{n+1}$ and $Y_{n+1}$:
\begin{eqnarray} \nonumber 
\bP(Y_{n+1} \in dy \mid Y_n = x) &=& \sum_{k=1}^\infty 
\bP(Y_{n+1} \in dy \mid  K_{n+1}=k, Y_n =x) \, \bP(K_{n+1}=k \mid Y_n=x) \\  \nonumber
&=&  \sum_{k=1}^\infty 
\bP(Y_{n+1} \in dy \mid  K_{n+1}=k) \, \bP(K_{n+1}=k \mid Y_n=x) \\  \nonumber
&=& \sum_{k=1}^\infty 
kG(y)^{k-1} g(y)dy  \, \bP(K_{n+1}=k \mid Y_n=x)
\end{eqnarray}
where we have used (\ref{YK}). Together with (\ref{KY}) we obtain 
\begin{eqnarray*}
Q(x,dy) &=& \sum_{k=1}^\infty\frac{(\lambda x)^k}{k!} e^{-\lambda x}
 k G^{k-1}(y)g(y)dy \;+\; e^{-\lambda x}g(y)dy \\
&=& e^{-\lambda x \big(1-G(y) \big)} 
\lambda xe^{-\lambda x}g(y)dy+e^{-\lambda x}g(y)dy
\end{eqnarray*}
and therefore the density of the transition kernel $Q(x,dy)$ with respect to the 
Lebesgue measure on $\bR$, $q(x,y)$, is given by (\ref{transition-density}).
\end{proof}
We next consider the process $\{K_n\}$ of the number of 
customers served in each stage.
\begin{proposition} \label{prop:Kn}
The sequence of number of customers served in each stage  $\{K_n\}$ is a 
discrete time Markov chain with state space $\bZ^+=\{1,2,\ldots\}$ and 
transition probability matrix given by 
\[
\bP(K_{n+1}=m\mid K_n=k)=\left\{
\begin{array}{lll}
 \int_0^\infty \frac{1}{m!}(\lambda x)^m e^{-\lambda x}kG^{k-1}(x)dG(x), 
&& \mbox{for } m=2,3,\ldots,\\  && \\
\int_0^\infty (\lambda x +1) e^{-\lambda x}kG^{k-1}(x)dG(x)  && \mbox{for }m=1.
\end{array} \right.
\]
\end{proposition}
\begin{proof}
Noting that $K_{n+1}$ depends solely on the number of Poisson arrivals during the $n$th
service stage, it follows that it is conditionally independent of $K_1, \ldots, K_n$, given
$Y_n$. is conditionally independent of $K_n$ given $Y_n$ we have 
\begin{eqnarray}  \nonumber
\bP(K_{n+1}=m\mid K_n=k,K_{n-1}=k_{n-1},\ldots,K_1=k_1) && \\ \nonumber
&& \hspace{-2.8in} =\; 
\int_0^\infty \bP(K_{n+1}=m, Y_n\in dy \mid K_n=k,K_{n-1}=k_{n-1},\ldots,K_1=k_1) \\ \nonumber
&& \hspace{-2.8in} =\; 
\int_0^\infty \bP(K_{n+1}=m \mid Y_n =y, K_n=k,\ldots,K_1=k_1)  \\  \nonumber
&& \hspace{-1in} \times \bP( Y_n\in dy \mid K_n=k,K_{n-1}=k_{n-1},\ldots,K_1=k_1) \\  \nonumber
&& \hspace{-2.8in} =\;  \int_0^\infty \bP(K_{n+1}=m \mid Y_n = y)\,  
\bP(Y_n\in dy \mid K_n=k) \\ \label{Kntrans}  
&& \hspace{-2.8in} =\;  
\bP(K_{n+1}=m\mid K_n=k) .
\end{eqnarray}
This shows establishes the markovian property for $\{K_n\}$. The transition
probability matrix for the chain is obtained immediately by using  (\ref{YK}) 
and (\ref{KY}) in (\ref{Kntrans}).
\end{proof}

\subsection{The service stage length}

In this section we will analyze the stationary behavior of the system by obtaining
the stationary distribution of the continuous state space Markov chain $\{Y_n\}$.
For concepts, terminology, and results pertaining to Markov chains on general 
state spaces we refer the reader to Tweedie \cite{Tweedie}, Meyn and Tweedie 
\cite{MT}, and Douc et al.\ \cite{Douc}. 
An alternative approach would entail analyzing first the discrete state space
Markov chain $\{K_n\}$, establish its positive recurrence, obtain its stationary distribution, 
and thence obtain corresponding results for the chain $\{Y_n\}$ of stage lengths. 
While we will not pursue this approach here, we will use it in a related problem for a 
gated $GI/M/\infty$ system in Section \ref{sc:SGMinf}.

Consider the measure $\phi$ on $({\sf X}, \cB)$ defined by 
$\phi(A) = \int_A g(y) dy$ which is absolutely continuous with respect to the Lebesgue 
measure on ${\sf X}=[0,\infty)$. This is an {\em irreducibility measure} for 
the chain $\boldsymbol{Y}:=\{Y_n\}$ because $\phi(A) > 0$ implies that 
$Q(x,A) >0$ for all $x\in {\sf X}$. In section \ref{sec:MCGeneralState} of the Appendix 
we show that $\boldsymbol{Y}$ is $\psi$-irreducible. ($\psi$-irreducibility is a property 
that generalizes the notion of irreducibility to general state Markov chains. 
The behavior of $\psi$-irreducible chains mimics that of irreducible chains 
with a countable state space in many respects. See \cite{Tweedie}, \cite{MT}.)

\begin{assumption} \label{ass:exp} We assume that the service time distribution satisfies the
Exponential Moment Condition
\[
\mbox{\em (EMC) } \quad \int_0^\infty e^{R y} g(y)dy < \infty \quad 
\;\;\;  \mbox{ for some } R>0. 
\]
Further we assume that the following holds.
\[
\mbox{ Light Traffic Condition: } \quad \lambda < R.
\]
\end{assumption}
In particular, when the service time distribution is exponential with rate $\mu$
the above Assumption is satisfied for any $R < \mu$ and $\lambda < R$ so that
$\rho :=\lambda /\mu <1$.
\begin{proposition}\label{prop:pi-moments} 
In the framework of this section, assume that the conditions of Assumption \ref{ass:exp}
hold. Then the Markov chain $\boldsymbol{Y}:=\{Y_n\}$ of the stage length process 
has a unique stationary probability measure $\pi$ with
density with respect to the Lebesgue measure on $\bR^+$, $f$, 
which satisfies the integral equation 
\begin{equation} \label{invariant-density}
f(y)=\int_0^\infty f(x) q(x,y)dx \hspace{0.2in} \mbox{ for all $y \in \bR^+$.}
\end{equation}
The stationary measure $\pi$ has finite $R$--exponential moments i.e.\
\begin{equation}  \label{exp-moment}
\int_0^\infty e^{R y} f(y)dy < \infty.
\end{equation}
\end{proposition}
The proof of this proposition is given is Section \ref{sec:MCGeneralState} of 
the Appendix. We point out that the restrictive exponential moment condition 
(EMC) is not necessary in order to prove the existence of a stationary measure with density
that satisfies (\ref{invariant-density}) but rather in order to show that this invariant 
density also satisfies the exponential moment condition (\ref{exp-moment}). Note that 
(EMC) implies the finiteness of moments of all orders of $g$ i.e.\
$\int_0^\infty y^n g(y) dy<\infty$ for all $n \in \bN$. As a consequence of Proposition 
\ref{prop:pi-moments} the moments of all orders of $f$ are also finite: 
$\beta_n := \int_0^\infty y^n f(y) dy < \infty$.

Equation (\ref{invariant-density}) is not easy to solve. Under Assumption \ref{ass:exp}
we will obtain here a solution in the form of series in $\lambda$. We begin 
by expressing the kernel density $q(x,y)$ as a power series in $\lambda$: 
From (\ref{transition-density}),
\begin{eqnarray*}
q(x,y) &=& \lambda x g(y) \sum_{k=0}^\infty (-1)^k \frac{(\lambda x)^k}{k!} 
\overline{G}(y)^k + g(y)\sum_{k=0}^\infty (-1)^k \frac{(\lambda x)^k}{k!}   
\end{eqnarray*}
which we rewrite as 
\begin{eqnarray}   \nonumber
q(x,y) &=& g(y)\, \left[ 1+ (\lambda x)^2 \left(\frac{1}{2}- \overline{G}(y) \right) 
- \frac{(\lambda x)^3}{2!} \left( \frac{1}{3} - \overline{G}(y)^2 \right) \right. \\
&& \hspace{2in} \left.+ \frac{(\lambda x)^4}{3!} \left( \frac{1}{4} 
- \overline{G}(y)^3 \right) - \cdots \right].   \label{fygx}
\end{eqnarray}
\begin{proposition} \label{prop:system_exp} Suppose that $\sigma_i$ are i.i.d.\ random variables 
with density $g$ satisfying the conditions of Assumption \ref{ass:exp}.
Let $f$ denote the invariant density of the stationary stage length, and suppose that 
$\beta_k:=\int_0^\infty x^k f(x) dx$, $k=1,2,\ldots$, denote its moments (which 
exist as shown in Proposition \ref{prop:pi-moments}). Define the quantities
\begin{equation} \label{gammadef}
\gamma_{m,k} \;:=\; \bE[\min(\sigma_1,\sigma_2,\ldots,\sigma_k)^m].
\end{equation}
Then the moments $\{\beta_i\}$, $i=2,3,\ldots,$ satisfy the infinite linear system
\begin{eqnarray}  \nonumber
\beta_i &=& \gamma_{i,1} + \frac{\lambda^2 \beta_2}{2!} \, \left( \gamma_{i,1}-\gamma_{i,2} \right) 
- \frac{\lambda^3 \beta_3}{3!} \, \left( \gamma_{i,1}-\gamma_{i,3} \right) +\cdots
+(-1)^j \frac{\lambda^j \beta_j}{j!} \, \left( \gamma_{i,1}-\gamma_{i,j} \right) 
+\cdots \\ \label{beta_system}
&& \hspace{4.1in} i=2,3,4,\ldots \;\;
\end{eqnarray}
The invariant density $f$ of the stage length can be expressed in terms of the moments 
$\{\beta_i\}$, the density $g$, and complementary distribution function $\overline{G}$ 
of the service time by the series
\begin{eqnarray}  \nonumber
f(y) &=& g(y)\,
 \left[ 1+ \frac{\lambda^2 \beta_2}{2!} \left(1- 2\overline{G}(y) \right) - \frac{\lambda^3 \beta_3}{3!} 
 \left( 1 - 3\overline{G}(y)^2 \right) \right. \\  \label{f-moments}
&& \hspace{2.3in} \left. + \frac{\lambda^4 \beta_4}{4!} \left( 1 - 4\overline{G}(y)^3 \right) - \cdots \right].
\end{eqnarray}
\end{proposition}
\begin{proof}
We begin with the relationship
\begin{equation}  \label{moments}
\int_0^\infty y^m g(y) \big( 1- k\overline{G}(y)^{k-1}\big) dy 
\;=\; \gamma_{m,1} - \gamma_{m,k} \;\;\;\;\;
\end{equation}
which holds because $kg(y)\overline{G}(y)^{k-1}$ is the density of 
$\min(\sigma_1,\sigma_2,\ldots,\sigma_k)$. We have, using using (\ref{invariant-density}),
\[
\beta_i \;=\; \int_0^\infty y^i f(y)dy \;=\;  \int_0^\infty y^i \left( \int_0^\infty f(x) q(x,y) dx \right) dy. 
\]
From this, taking into account (\ref{fygx}), we have
\begin{eqnarray} \nonumber
\beta_i &=& \int_0^\infty y^i \, 
\left( \int_0^\infty \left[ 1+ \sum_{k=2}^\infty  (-1)^k\frac{(\lambda x)^k}{k!}
\left( 1- k \overline{G}(y)^{k-1}\right) \right] f(x)dx \right)  g(y)\,dy  \\ \label{first-interchange}
&=& \int_0^\infty y^i \, \left[ 1+ \sum_{k=2}^\infty  (-1)^k\frac{\lambda^k \beta_k}{k!}
\left( 1- k \overline{G}(y)^{k-1}\right) \right]  g(y)\,dy
\end{eqnarray}
where we have interchanged the second integral with the sum. This
is justified by the fact that
\[
 \sum_{k=2}^\infty  \int_0^\infty  \frac{(\lambda x)^k}{k!}
\left| 1- k \overline{G}(y)^{k-1}\right|  f(x)dx \; \leq \; 
\sum_{k=2}^\infty \frac{\lambda^k \beta_k}{k!}  + 
\sum_{k=2}^\infty \frac{\lambda^k \beta_k}{k!} k\overline{G}(y)^{k-1}.
\]
The first sum on the right is finite (and in fact equal to 
$\bE [e^{\lambda Y}] - \lambda \bE[Y] -1$) by virtue of (\ref{exp-moment})
and the light traffic condition $\lambda < R$.
The second sum can also be seen to be finite for every $y>0$ because its terms are bounded
in absolute value by  $\lambda\sum_{k=2}^\infty \frac{1}{k!} \beta_k k \lambda^{k-1}$.
This sum is finite because the power series $\sum_{k=2}^\infty \frac{1}{k!} \beta_k \lambda^{k}$ 
in $\lambda$ has radius of convergence $R$ and may therefore be differentiated term by term
at $\lambda<R$.

Next, we may interchange the integral and summation in (\ref{first-interchange}) because
 \[
 \sum_{k=2}^\infty  \frac{\lambda^k \beta_k}{k!}
\int_0^\infty y^i \left( 1+ k \overline{G}(y)^{k-1}\right)dy =
\sum_{k=2}^\infty  \frac{\lambda^k \beta_k}{k!} \left( \gamma_{i,1}+\gamma_{i,k} \right)
\;<\; 2 \gamma_{i,1} \, \bE[e^{\lambda Y}] \;<\; \infty.
 \]
Therefore interchanging again integration and summation we obtain
\begin{eqnarray*}
\beta_i 
&=& \int_0^\infty y^i g(y)\,dy \;+\; \sum_{k=2}^\infty  (-1)^k
\frac{\lambda^k \beta_k}{k!}\int_0^\infty y^i g(y)  
\left( 1- k \overline{G}(y)^{k-1}\right)\,dy .
\end{eqnarray*}
Taking into account (\ref{moments}), we obtain (\ref{beta_system}). \end{proof}

\begin{figure}[h!]
  \begin{center}
    \includegraphics[width=6in]{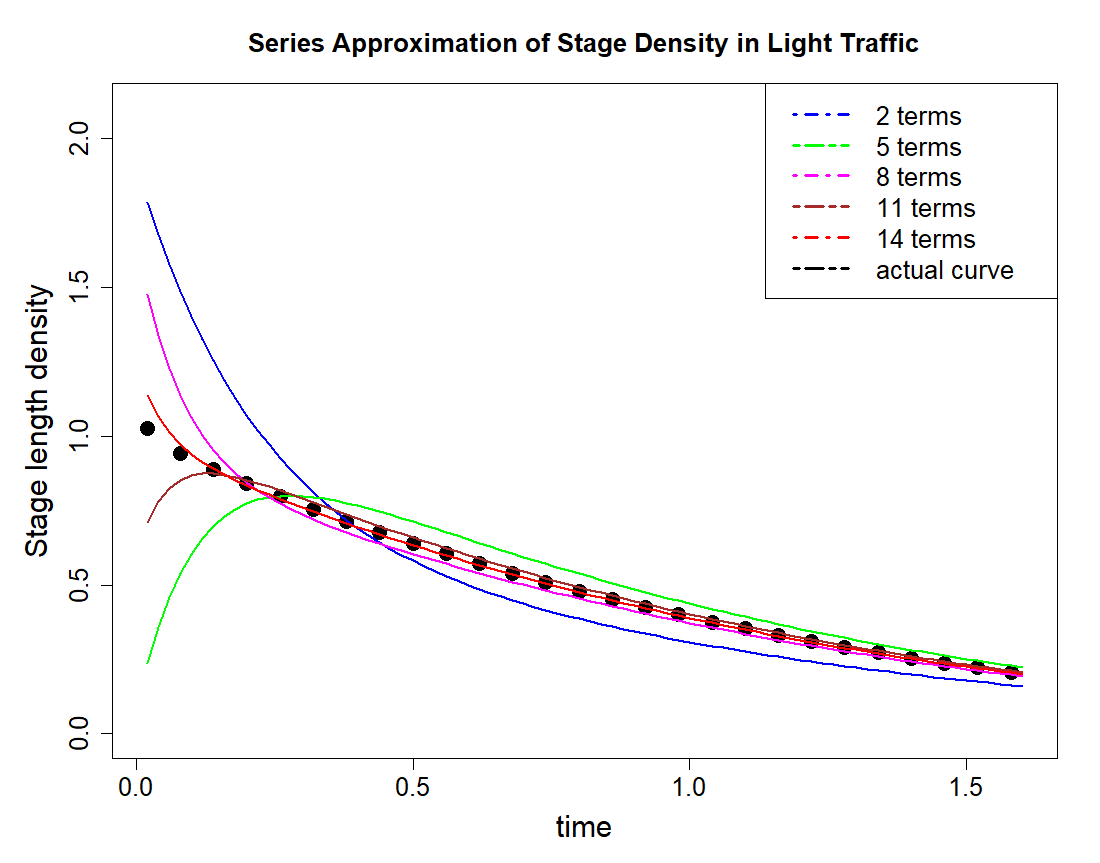}
  \end{center}
  \caption{A series representation for the stationary density of the length of a stage 
  in a gated $M/M/\infty$ queue in light traffic. Here $\lambda=1$ and $\mu =1.2$ 
  corresponding  to $\rho=0.833$. A   truncated ($20 \times 20$) version of the 
  system (\ref{beta_system}) is solved (by Jacobi iteration) to obtain approximate 
  values for $\beta_2,\ldots,\beta_{21}$. These in turn are used in (\ref{f-moments}) 
  to obtain the successive approximations given here.}
  \label{fig2}
\end{figure}

As we will show in the sequel the moment sequence $\{\beta_i\}$ can be determined 
from the above system and the invariant density can then be obtained from the series 
(\ref{f-moments}). The discussion regarding the existence and uniqueness of the 
solution of the above infinite system, 
as well as its approximation by considering a truncated version of (\ref{beta_system}) 
and the convergence of the infinite system (\ref{f-moments}) will be discussed in the sequel.
\begin{remark}
The above system (\ref{beta_system}) does not determine the first moment of the stationary 
stage length, $\beta_1$. Note however that the expression for stationary density $f$ in 
(\ref{f-moments}) is expressed in terms of $\beta_2,\beta_3,\ldots$ and does not involve 
$\beta_1$ which can therefore be obtained as 
\begin{eqnarray}  \nonumber
\beta_1 &=& \int_0^\infty f(y) y dy \;=\; \int_0^\infty g(y) y dy \;+\; \sum_{k=2}^\infty  (-1)^k
\frac{\lambda^k \beta_k}{k!} \int_0^\infty y\big( 1 -k\overline{G}(y)^{k-1} \big)g(y)  dy \\   \label{beta1}
&=& \gamma_{1,1} \;+\; \sum_{k=2}^\infty  (-1)^k\frac{\lambda^k \beta_k}{k!}
\big( \gamma_{1,1} -\gamma_{1,k}\big).
\end{eqnarray}
Figure \ref{fig2} gives consecutive partial sums of the series (\ref{f-moments}) 
when $g(y) = \mu e^{-\mu y}$.
\end{remark}
\begin{remark}
If the service time density $g$ is has bounded support, or if $g$ has unbounded support 
but superexpoential tails (as has for instance the Rayleigh density
$g(y) = \frac{y}{\sigma^2} e^{-\frac{y^2}{2\sigma^2}}$, $y \geq 0$, $\sigma>0$),
then it satisfies the exponential moment condition of Assumption \ref{ass:exp} for any 
$R>0$ and therefore  the condition $\lambda < R$ is satisfied for any $\lambda >0$. 
The moment sequence of the stationary service stage length satisfies then 
system (\ref{beta_system}) for any $\lambda>0$. 
In particular, if $K_g :=\sup\{x>0 : \int_0^x g(y) dy < 1\}$ denotes the upper limit of the support 
of $g$ then, when $K_g<\infty$, the approach based on (\ref{beta_system}) and the corresponding
approximation is valid. This is illustrated in section \ref{sec:unif}. 
\end{remark}

\subsection{The infinite system for the service stage length} \label{sec:MG1}
Define
\begin{equation} \label{nudef}
\nu_i:= \frac{\lambda^i}{i!} \beta_i,   \quad  c_i:=\frac{\lambda^i}{i!} \gamma_{i,1},  
\quad \mbox{and} \quad \Gamma_{ij}:= 
\frac{\lambda^i}{i!}\big( \gamma_{i,1} - \gamma_{i,j}\big), \qquad i,j=2,3,\ldots.
\end{equation}
Note that these quantities, including $\Gamma_{ij}$, are positive because,
due to its definition, $\gamma_{i,1} > \gamma_{i,j}$ for all $i,j$. Let also 
$\ell_1:=\big\{x=(x_i)_{i=2,3,\ldots} : \sum_{i=2}^\infty |x_i|<\infty \big\}$
denote the space of absolutely summable real sequences indexed by the set $\{2,3,\ldots\}$. 
$\Vert x \Vert :=\sum_{i=2}^\infty |x_i|$ denotes the 1-norm and we use the same symbol
for the induced norm for bounded linear operators on $\ell_1$.
Consider the system 
\begin{equation} \label{mod_system}
x_i \;=\; c_i + \sum_{j=2}^\infty \Gamma_{ij} x_j, \hspace{0.2in} i=2,3,\ldots.
\end{equation}
The sequence $\{\nu_i\}_{i=2,3,\ldots}$ is an element on $\ell_1$ corresponding to the 
positive solution of the system (\ref{beta_system}) which may be transformed into
\begin{equation} \label{nu_mod_system}
\nu_i \;=\; c_i + \sum_{j=2}^\infty \Gamma_{ij} \nu_j, \hspace{0.2in} i=2,3,\ldots
\end{equation}
by multiplying the $i$th line of the system  by $\frac{\lambda^i}{i!}$ 
for each $i=2,3,\ldots$. Thus we see from Proposition \ref{prop:system_exp} that $\{\nu_i\}$ is a positive
solution of (\ref{mod_system}). In order to study the system (\ref{mod_system}) and
answer the question of the uniqueness of the solution we consider the {\em positive} linear operator 
$\Gamma: \ell_1 \rightarrow \ell_1$ defined by 
$\big( \Gamma z \big)_i = \sum_{j=2}^\infty \Gamma_{ij} z_j$, $i=2,3,\ldots$ for all
$z \in \ell_1$.
\begin{theorem} \label{th:MGsystem}
Under the conditions of Assumption \ref{ass:exp}, $\nu$ defined in (\ref{nudef}) is the unique solution 
of the system (\ref{mod_system}) in $\ell_1$. Further, there exits $\theta \in (0,1)$ such that 
\begin{equation} \label{Gamma_norm_ineq}
\Vert \Gamma^n  \Vert \leq \Vert \nu \Vert \, \theta^{n-1}, \;\; n=1,2,\ldots,
\end{equation}
the operator $I -\Gamma$ is invertible in $\ell_1$, and the unique solution of
(\ref{mod_system}) is given by $\nu = (I-\Gamma)^{-1} c$. 
\end{theorem}
\begin{proof}
The fact that the positive sequence $\nu$ is a solution of the system 
(\ref{mod_system}) is a consequence of Propositions \ref{prop:pi-moments} 
and \ref{prop:system_exp}. To show that this solution belongs to $\ell_1$ note that
\[
\sum_{i=2}^\infty \nu_i = \sum_{i=2}^\infty \frac{\lambda^i}{i!} \beta_i = 
\int_0^\infty e^{\lambda y}f(y) dy - \lambda \int_0^\infty y f(y)dy - 1 <\infty.
\] 
Similarly, $c \in \ell_1$ because
\[
\sum_{i=2}^\infty c_i = \sum_{i=2}^\infty \frac{\lambda^i}{i!} \gamma_{i,1} = 
\int_0^\infty e^{\lambda y}g(y) dy - \lambda \int_0^\infty y g(y)dy - 1 <\infty.
\] 
We next show that the existence of a positive solution $\nu \in \ell_1$ of 
the system (\ref{mod_system}) implies that (\ref{Gamma_norm_ineq}) holds for 
the positive linear operator $\Gamma$.
Using the inequality $0<\Gamma_{ij} \leq c_i$ for all $i,j$, in (\ref{mod_system}) we obtain
\[
\nu_i  \; \leq \;  c_i +  \sum_{j=2}^\infty c_i \nu_j \;=\;
  c_i \Big( 1+\sum_{j=2}^\infty \nu_j\Big) \; \leq \; c_i \big(1+ \Vert \nu \Vert \big),
\]
whence we get
\begin{equation} \label{nu_ineq}
 \frac{\nu_i}{1+ \Vert \nu \Vert}  \; \leq \;  c_i,   \; \; i=2,3,\ldots .
\end{equation}
Using (\ref{nu_ineq}) we obtain from  (\ref{mod_system})
\[
\sum_{j=2}^\infty \Gamma_{ij} \nu_j \;= \; \nu_i - c_i \;\leq \; 
\nu_i - \frac{\nu_i}{1+ \Vert \nu \Vert} \;=\; 
\nu_i \, \theta , \; \; i=2,3,\ldots,
\]
where
\begin{equation} \label{Norm_theta}
\theta:= \frac{\Vert \nu \Vert}{1+ \Vert \nu \Vert}<1.
\end{equation}
Thus we may write $\Gamma \nu \leq \theta \nu$ and by virtue of the positivity of both 
$\Gamma$ and $\nu$ we obtain inductively
\begin{equation}  \label{Gamma-n}
\Gamma^k \nu \;\leq \; \theta^k \nu , \;\;\; k=1,2,\ldots .
\end{equation}
The inequality above holds in the componentwise sense. For any $z \in \ell_1$ and $n \in \bN$
\begin{equation}  \label{norm-intermed} 
\big\vert \big(\Gamma^n z\big)_i \big\vert \;  \leq  \; \sum_{j=2}^\infty \Gamma^n_{ij} | z_j | 
\;=\; \sum_{k=2}^\infty \Gamma_{ik}^{n-1}\sum_{j=2}^\infty \Gamma_{kj} | z_j | 
\; \leq \;  \sum_{k=2}^\infty \Gamma_{ik}^{n-1} \sum_{j=2}^\infty c_k  |z_j|.
\end{equation}
Using the inequalities $c_k \leq \nu_k$, $k=2,3,\ldots$, the fact that 
$\Vert z \Vert = \sum_{j=2}^\infty |z_j|$, and (\ref{Gamma-n}) in the last term of the above
string of inequalities we obtain
\[
\big\vert \big(\Gamma^n z\big)_i \big\vert \;  \leq  \; 
\Vert z \Vert  \, \sum_{k=2}^\infty \Gamma_{ik}^{n-1}  \nu_k  \; \leq \; 
\Vert z \Vert  \, \theta^{n-1} \nu_i.
\]
Since $\Vert \Gamma^n z \Vert 
= \sum_{i=2}^\infty \big\vert \big(\Gamma^n z\big)_i \big\vert$,
from (\ref{norm-intermed}) we have
\[
\Vert \Gamma^n z  \Vert \; \leq \; \theta^{n-1}  \Vert z \Vert \; \Vert \nu \Vert, 
\;\;\; n=1,2,\ldots.
\]
which implies (\ref{Gamma_norm_ineq}). This in turn implies that 
$\limsup_{n\rightarrow \infty} \Vert \Gamma^n \Vert^{1/n} <1$. Hence the
Neumann series $\sum_{n=0}^\infty \Gamma^n$ converges and $I-\Gamma$ (where $I$ 
is the identity transformation on $\ell_1$) is invertible. (See for instance Theorem 5.7 of 
\cite[p. 146]{Alt}.) Its inverse is $(I-\Gamma)^{-1} = \sum_{n=0}^\infty \Gamma^n$ and 
we conclude that the solution $\nu$ of the system (\ref{mod_system}) 
is unique and given by 
\[
\nu = (I-\Gamma)^{-1} c.
\]
Note that $R:=(I-\Gamma)^{-1}$ is a positive operator. 
\end{proof}
In practice one would approximate the infinite linear system (\ref{mod_system})
by keeping only a finite number of equations and unknowns. It is therefore of interest to
establish a result that, in some sense, the solution arising from a sufficiently large number 
of equations and variables approximates well the solution of the infinite system.
Such a truncation is often referred to in the literature as a {\em finite section} and the process
of finite approximations as the {\em projection method}. We refer the reader to \cite{Kress} 
and \cite{Lindner}. The next section establishes a precise statement in this direction.

\begin{figure}[!h]
  \begin{center}
    \includegraphics[width=5.4in]{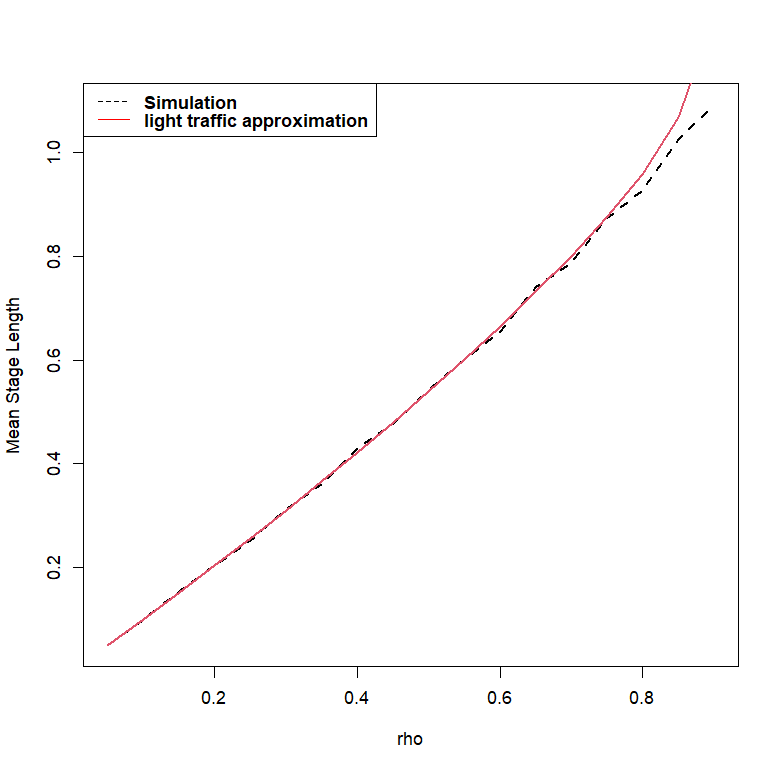}
  \end{center}
  \caption{The mean stage length in the Gated $M/M/\infty$. The dotted line is 
  obtained by simulating $10^4$ stages. The red line, gives the light
  traffic approximation (equation \ref{Exp-MeanStageLength}). As the traffic 
  intensity $\rho$ approaches 1, the quality of the light traffic approach deteriorates. The
  approximation uses a truncated system $15\times 15$ whose solution is obtained by
  iteration.}
  \label{fig4}
\end{figure}

\subsection{The Finite Section Method}
We next show that the finite system resulting from (\ref{mod_system}) by 
keeping the first $N$ equations and the first $N$ unknowns has a unique
solution and the sequence of solutions converges to the solution of the
infinite system. The finite $N\times N$ system is often referred to as a {\em
finite section} of the infinite system. 

Let $P_N: \ell_1 \to \ell_1\) be the canonical truncation (projection) operator 
defined component-wise for any \(x = (x_2, x_3, \dots) \in \ell_1\) by:
\[
\big( P_{N}x\big)_{i}=\begin{cases}x_{i} &\text{if \ }2\le i\le N+1\\ 0&\text{if \ }i>N+1\end{cases}
\]
$P_{N}$ is a bounded linear operator with operator norm $\Vert P_N\Vert = 1$ for all 
$N \ge 1$ and the range of $P_{N}$, $\mathcal{R}(P_N)$, is an $N$-dimensional 
subspace of $\ell_1$ isometric to $\mathbb{R}^{N}$. For any $x \in \ell_1$, the sequence $P_{N}x$ 
converges to $x$ in the $\ell_{1}$-norm as $N \to \infty$. 
Define the ``truncated vectors'' $c^N := P_N c$ and $x^N := P_N x^N$. 
In practice when we truncate the system retaining the first $N$ unknowns and the first $N$ 
equations we deal with elements in $\mathbb{R}^N$. Here we will find it more convenient 
to consider elements in the range on $P_N$ which are elements in $\ell_1$ all of whose 
components except the first $N$ vanish. In this fashion all the elements of the sequence $\{x^N\}$
of successive approximation belong to the same space.  We define $\Gamma^N: \ell_1 \to \ell_1$ as
$\Gamma^{N}:=P_{N}\Gamma P_{N}$. When restricted to $\mathcal{R}(P_N)\), \(\Gamma^{N}$ 
behaves precisely as the $N \times N$ matrix $(\Gamma_{ij})_{1\leq i,j \leq N}$ acting on 
an $N$-dimensional vector.
\begin{theorem} \label{th:MG_finite_section}
Under the conditions of Assumption \ref{ass:exp}, for each $N$ the principal finite system 
$x^N = c^N + \Gamma^N x^N$ has a unique solution 
$\nu^N$. Furthermore for each $N$ we have the error estimate
\begin{equation} \label{MG_error}
\Vert \nu - \nu^N \Vert  \; \leq \; B\, \|(I-P_N) \nu\|
\end{equation}
where $B:= \Big(1+ \Vert \nu \Vert \Big)\, \Big( 1 +\Vert \nu \Vert +\Vert \nu \Vert^2 \Big)$,
and $\lim_{N\rightarrow \infty}\Vert \nu^N -\nu \Vert =0$.
\end{theorem}
\begin{proof}
Because $\Gamma_{ij}>0$ for all $i,j$, if follows that 
$\Vert (\Gamma^N)^k \Vert \leq \Vert \Gamma^k \Vert \leq \Vert \nu \Vert \theta^{k-1}$
for all $k \in \bN$. Therefore we see again that the Neumann series converges:
\[
\Big\|\sum _{k=0}^\infty (\Gamma^N)^k \Big\|
\; \le \;  \sum_{k=0}^\infty \|\big( \Gamma^N\big)^k \|
\; \le \;  1+\sum _{k=1}^\infty \Vert \nu \Vert \, \theta^{k-1} 
\;=\; 1+ \frac{\Vert \nu \Vert}{1-\theta } .
\]
This implies that $I-\Gamma^N$ is invertible for each $N$ and its norm is uniformly bounded
in $N$: 
\[
\Vert (I -\Gamma^N)^{-1} \Vert \; \leq \; 1+ \frac{\Vert \nu \Vert}{1-\theta}.
\]
We next analyze the size of the truncation error.
\begin{eqnarray*}
\nu-\nu^N &=& \nu-(I-\Gamma^N)^{-1}P_N c \;=\; 
(I-\Gamma^N)^{-1}\left[ \nu -\Gamma^N\nu -P_N c \right]  \\
&=& (I-\Gamma^N)^{-1} \left[\nu - \Gamma^N \nu - P_N \nu + P_N \Gamma \nu  \right]
\end{eqnarray*}
where we have used the fact that $c  = \nu - \Gamma \nu$.
Recalling that $\Gamma^N = P_N \Gamma P_N$ we get
\[
\nu - \nu^N = (I-\Gamma^N)^{-1} \left[(I-P_N) \nu + P_N \Gamma (I-P_N) \nu \right].
\]
Taking norms and applying the triangle inequality we obtain
\begin{eqnarray*}
\|\nu -\nu^N\|  &\le & \left\|(I-\Gamma^N)^{-1}\right\|  \Big( \|(I-P_N)\nu\|
+\|P_N\Gamma(I-P_N)\nu \| \Big) \\
&\le&  \Big( 1+ \frac{\Vert \nu \Vert}{1-\theta }\Big)  
\Big( \|(I-P_N) \nu \| +\|\Gamma\|   \|(I-P_N) \nu \| \Big) \\
&=&  B \, \|(I-P_N) \nu \| 
\end{eqnarray*}
where in the last equality we have used (\ref{Gamma_norm_ineq}) with $n=1$ to
bound $\Vert \Gamma \Vert$, (\ref{Norm_theta}), and the definition of $B$. 
Letting $N\rightarrow \infty$ and taking into account that 
$(I-P_N)\nu \rightarrow 0$ since $\nu \in \ell_1$ proves that 
$\nu^N \rightarrow \nu$ as $N \rightarrow \infty$ in $\ell_1$.
\end{proof}
Note that (\ref{MG_error}) is useful in that it establishes the convergence of the sequence
of the solutions of finite sections to the solution. However the value of $B$ is not known
in advance since it depends on the unknown solution $\nu$.


\begin{figure}[!htbp]
\vspace{-0.8in}
	\begin{center}
		\includegraphics[width=5.8in]{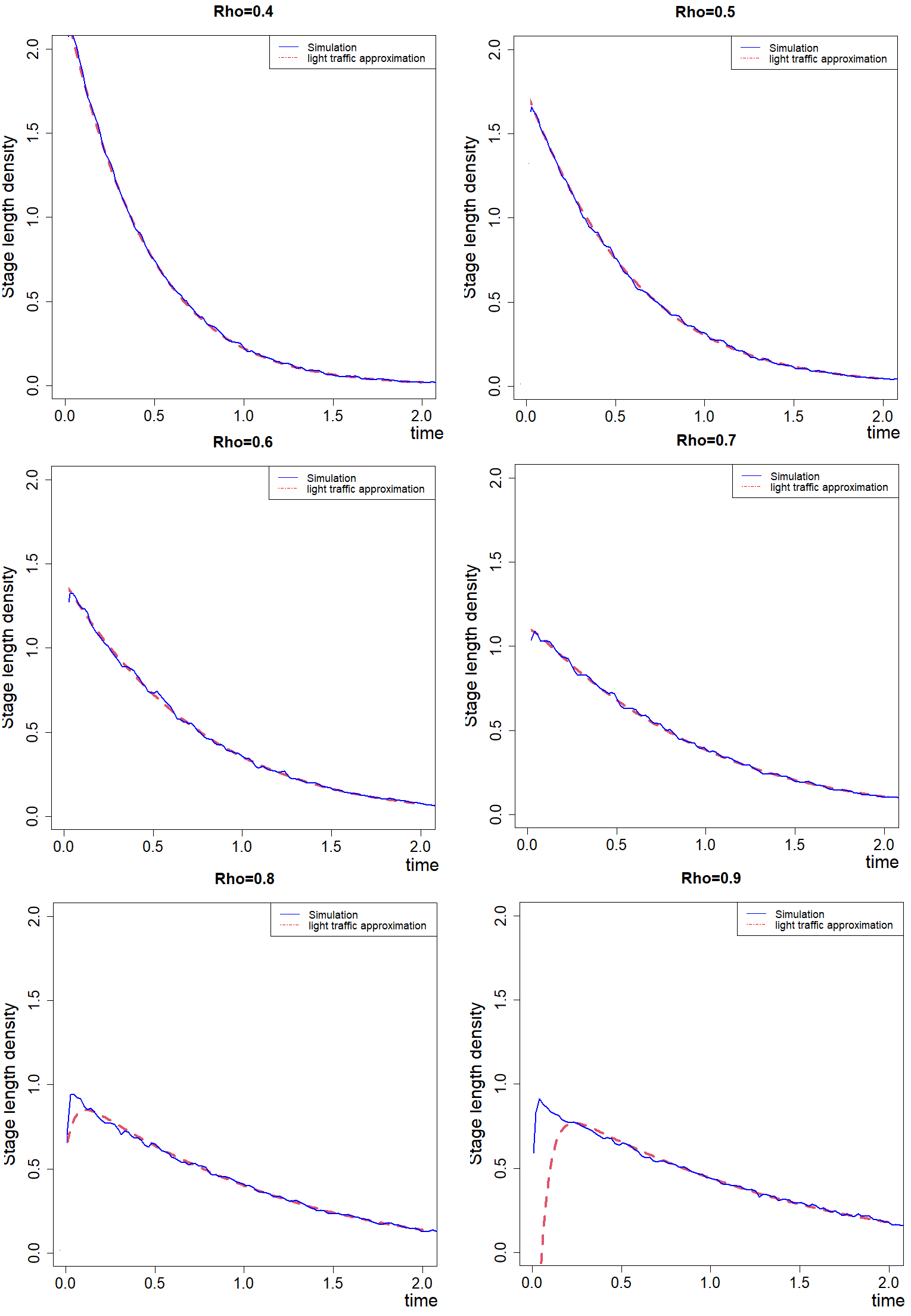}
	\end{center}
	\vspace{-0.2in}
	\caption{A comparison of the actual invariant density of the service stage length of an 
	$M/M/\infty$ gated queue with the light traffic approximation. 
	The finite section is $29\times 29$ 
	and the finite linear system is used by iteration. The approximation uses the obtained estimates
	of the moments $\beta_2, \ldots, \beta_{30}$ in (\ref{invariant-density-m}). The simulation uses the
	inverse transform method to generate consecutive service stage lengths using the transition probability
	$\bP(Y_{n+1}\leq y | Y_n=x)= \exp(-\lambda x e^{-\mu y}) - \exp(-\mu y +\lambda x)$. We generate
	$10^5$ stage lengths and use an initial runin of $10^4$. The density kernel estimator of {\tt R} is used 
	with	a Gaussian kernel with smoothing window $bw=0.008$. }
	\label{fig3}
\end{figure}

\subsection{The infinite linear system (\ref{beta_system}) for exponential service times}
We continue the above analysis by examining the special case where the service time 
distribution is exponential, i.e.\ $g(x) := \mu e^{-\mu x}$, $x \geq 0$. Then the service time
distribution has exponential moments of any order $R< \mu$ and the light traffic condition
is $\lambda < R$, or equivalently $\rho:=\lambda/\mu < 1$. 
Since the minimum of $j$ independent exponential random variables 
$\min(\sigma_1,\ldots,\sigma_j)$ with $\sigma_l \sim \mbox{exp}(\mu)$ is also 
exponential with rate $j\mu$ we have
\[
 \gamma_{i,1} = \frac{i!}{\mu^i}, \hspace{0.2in}  \gamma_{i,j} = \frac{i!}{j^i \mu^i}\;.
\]

Taking this into account the infinite system of linear equations (\ref{beta_system}) is written as
\begin{eqnarray}  \nonumber
\beta_i &=& \frac{i!}{\mu^i} + \frac{\lambda^2 \beta_2}{2!} \, 
\left( \frac{i!}{\mu^i}-\frac{i!}{2^i \mu^i} \right) - \frac{\lambda^3 \beta_3}{3!} \, 
\left( \frac{i!}{\mu^i}-\frac{i!}{3^i \mu^i} \right) 
+\cdots+(-1)^j \frac{\lambda^j \beta_j}{j!} \, \left( \frac{i!}{\mu^i}-\frac{i!}{j^i \mu^i} \right) +\cdots \\  
\label{m-infinite-system}
&& \hspace{4in} i=2,3,\ldots \;\;
\end{eqnarray}
Letting again 
$\nu_i := \frac{\lambda^i \beta_i}{i!}$,  $i=2,3,\ldots$, 
multiplying both sides of (\ref{m-infinite-system}) by $\frac{\lambda^i}{i!}$, and setting 
$\rho:=\lambda/\mu$, 
we obtain the system
\begin{equation} \label{y_system}
\rho^{-i} \nu_i \;=\; 1 + \nu_2 \left( 1- 2^{-i}\right) - \nu_3 \left( 1- 3^{-i}\right) + \cdots 
+(-1)^j  \nu_j \left( 1- j^{-i}\right) + \cdots,   \hspace{0.1in} i=2,3,\ldots .
\end{equation}
Since $\rho<1$ Theorem \ref{th:MGsystem} guarantees the existence of a unique positive
element $\nu \in \ell_1$ which satisfies (\ref{y_system}). Having obtained $\nu$, 
the invariant stage length density is given by (\ref{f-moments}) as
\begin{equation} \label{invariant-density-m}
f(x) \;=\; \mu e^{-\mu x} + \sum_{k=2}^\infty  (-1)^k \nu_k  
\left( \mu e^{-\mu x} - k \mu e^{-k\mu x}\right) , \;\;\; x>0.
\end{equation}
Figure \ref{fig3} gives an approximation based on a finite section of (\ref{y_system})  and 
(\ref{invariant-density-m}) and shows the comparison with simulation results for a range of 
values of $\rho <1$.

From (\ref{y_system}) and (\ref{invariant-density-m}) we may also obtain the first moment 
of the stationary stage length distribution
\begin{equation} \label{Exp-MeanStageLength}
\beta_1 \;=\; \int_0^\infty x f(x)dx \;=\; \frac{1}{\mu} + 
\frac{1}{\mu} \sum_{k=2}^\infty  (-1)^k \frac{\lambda^k \beta_k}{k!} \, \frac{k-1}{k} 
\;=\;\frac{1}{\mu} + \frac{1}{\mu} \sum_{k=2}^\infty  (-1)^k \nu_k \, \frac{k-1}{k}  .
\end{equation}
Figure \ref{fig4} compares an approximation of the mean stage length based on a partial
sum of the above series to simulation results.


\subsection{The infinite linear system (\ref{beta_system}) for uniform service times}  \label{sec:unif}
Suppose the service time density is uniformly distributed in $[0,1]$. Then the exponential moment
Assumption \ref{ass:exp} (and light traffic condition) is satisfied for any $R>0$ {\em and therefore 
for any $\lambda >0$}. An easy computation shows that
\[
\gamma_{i,1} =  \frac{1}{i+1} , \qquad \gamma_{i,j} = \frac{i! j!}{(i+j)!}. 
\]
With $\nu_i = \frac{\lambda^i \beta_i}{i!}$ the infinite linear system for the moments 
(\ref{beta_system}) becomes
\begin{equation} \label{nu_unif}
\nu_i \;=\; \frac{\lambda^i}{(i+1)!} + \sum_{j=2}^\infty \nu_j (-1)^j \lambda^i \left( \frac{1}{(i+1)!} - \frac{j!}{(i+j)!}\right),
\quad i=2,3,\ldots.
\end{equation} 
From the results of section \ref{sec:MG1} this has  a unique solution for any $\lambda >0$. 
The stationary service stage density is given by
\begin{equation} \label{density_unif}
f(y) \;=\; 1 + \sum_{j=2}^\infty \nu_j (-1)^j \left( 1- j(1-y)^{j-1} \right) \;\;\; y \in [0,1].
\end{equation}
Figure \ref{fig1} illustrates the fact that, when the service time distribution has exponential moments 
of all orders (or equivalently when its moment generating function is an entire function) then the 
method we proposed can be used to obtain numerical solutions for any value of $\lambda$.

The system discussed here was first considered in \cite{Browne2} from the point of view of 
the number of customers served in a stationary stage. If $p_k$ is the probability that a service stage
serves $k$ customers in stationarity and $P(y)= \sum_{k=1}^\infty p_k y^k$ is the corresponding
pgf, it is shown there that $Q(y):= 1- P(y)$ satisfies the Fredholm integral equation
of the second type
\[
Q(y)  \; =\; \int_0^1 K(y,t) Q(t)dt  + 1-y
\]
with kernel $K(y,t) = \lambda (1-y) [ e^{-\lambda (1-y) t} - e^{-\lambda t } ]$. Based on
this, \cite{Browne2} obtains light and heavy traffic asymptotics. Expanding the kernel in series 
the above integral equation can be transformed in the system $x = c +A x$. Then a detailed analysis 
of the system is given in $l_2$ and a solution is obtained by iteration as $x = \sum_{r=0}^\infty A^r c$ 
since it is shown that $\Vert A \Vert_2 < 1$. The tail of the probability distribution for the number of 
customers in a service stage is then given in terms of the solution $x$ as
\[
\sum_{i=k+1}^\infty p_i = \lambda \left( \frac{\lambda^k}{k!} \right)^{1/2} x_k, \quad k=1,2,\ldots.
\]
Detailed bounds for the truncation error are also given.

\begin{figure}[h!]
  \begin{center}
    \includegraphics[width=6.0in]{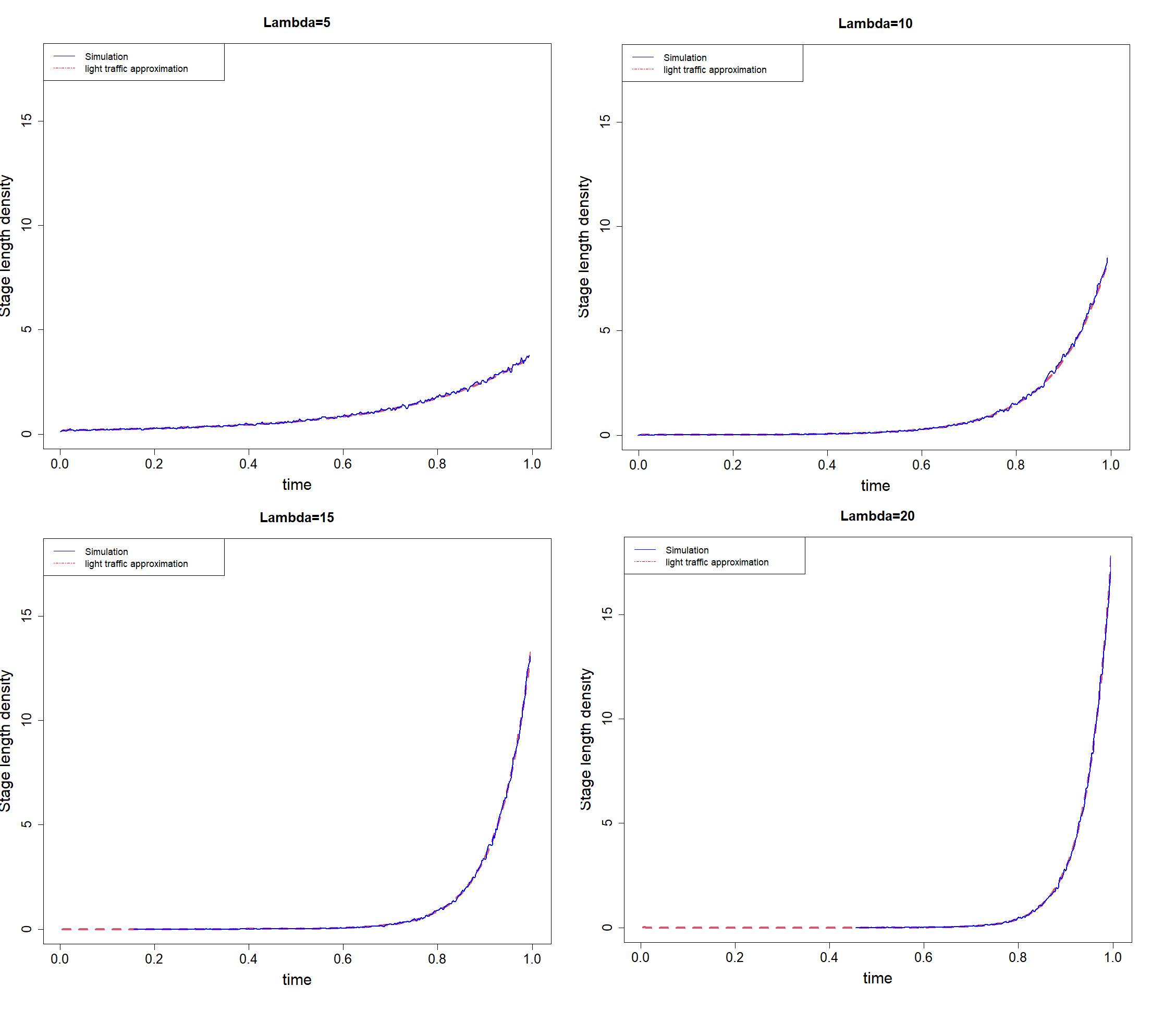}
  \end{center}
  \caption{The service stage length density when service times are uniform in $[0,1]$ for various values 
               of $\lambda$. The approximation uses a finite section of the system with $N=100$ terms. 
               (This large value of $N$ is necessary only when $\lambda =20$.) The simulation results use 
               $10^5$ stages.}
  \label{fig1}
\end{figure}


\subsection{The number of customers served in a service stage}  \label{sec:number_k}

Having determined the stationary density of the service stage length in (\ref{f-moments}) 
we may use it to obtain the distribution of the number of customers served in a stage in 
stationarity using (\ref{KY}) as  
\begin{eqnarray}  \label{PKK}
\bP(K=k) &=& \frac{\lambda^k}{k!} \sum_{m=0}^\infty \frac{(-\lambda)^m}{m!} \beta_{k+m}, 
\qquad  k=2,3,\ldots, \\  \nonumber
\bP(K=1) &=& \sum_{m=0}^\infty \frac{(-1)^m}{m!}\left(\lambda \beta_{m+1} + \beta_m\right).
\end{eqnarray}
In particular, from (\ref{KY}),
\begin{eqnarray*}
\bE[K_{n+1} | Y_n] &=& \sum_{k=2}^\infty k\frac{(\lambda Y_n)^k}{k!} e^{-\lambda Y_n}\;+\;
(1+\lambda Y_n) e^{-\lambda Y_n} \;=\; \lambda Y_n + e^{-\lambda Y_n}  \\
&=& 1 + \sum_{k=2}^\infty (-1)^k \frac{\lambda^k}{k!} \, Y_n^k.
\end{eqnarray*}
Taking expectation with respect to $Y_n$ inside the summation, (which is justified under our assumptions) 
we conclude that the expected number of customers served in a stage in stationarity is
\begin{equation}
\bE[K] \;=\; 1 + \sum_{k=2}^\infty (-1)^k \frac{\lambda^k}{k!} \beta_k.
\end{equation}


\section{The synchronized gated $GI/M/\infty$ system - number of customers 
served in a stage}  \label{sc:SGMinf}

We now turn to an infinite server system, related to that of 
section \ref{sec:GQ}, in which customers arrive according to a renewal process 
$\{S_m\}$ and have service requirements $\{\sigma_m\}$ that are i.i.d.\ random variables,
exponentially distributed with rate $\mu$, independent of the renewal arrival process. 
The gate mechanism operates essentially in the same fashion 
as in the $M/GI/\infty$ system considered in section \ref{sec:GQ} {\em with the 
exception that when the customers in a stage (who are again served in parallel) 
complete service and leave, the gate remains closed until the next arrival 
epoch.} At that point, any customers that have arrived during the service stage and
are in the waiting area, together with the newly arrived customer, are all admitted 
for service. This choice of policy results in a model that is analytically tractable because the 
epochs of initiation of batch service are a subset of the renewal epochs. It is an appropriate model
when a synchronizing mechanism, e.g.\ a shuttle or a polling signal, controls both the arrival 
of customers and the departure of processed batches as well as the operation of the gate which
admits a batch of waiting customers to the processing station. A model with such a 
synchronizing mechanism describing the arrival and departure of customers by a shuttle
has been discussed in \cite{Zazanis}.

We consider the synchronous version of the renewal arrival process with $S_0=0$
and interarrival times $\{\tau_m\}$, $m=1,2,\ldots$, that are i.i.d.\ positive 
random variables with common distribution $B$. Thus $S_m= S_{m-1}+\tau_m$.
Let $\{T_n\}$ denote the epochs of service stage initiations when the gate 
momentarily opens. The beginning of a service 
stage always coincides with an arrival epoch of a customer and hence the
point process $\{T_n\}$ is a subset of the arrival process $\{S_m\}$. Assuming 
that at time $T_0=0$ there are $K_0$ customers
at the gate, including the customer who arrives at time $S_0=0$ 
(using a slightly different numbering scheme from that of section \ref{sec:GQ}), the first 
service stage serves $K_0$ customers and is concluded at time 
$T_1:=\min\big\{S_m: S_m>\max \big(\sigma_1,\ldots,\sigma_{K_0}\big) \big\}$. 
The dynamics of the system can be described for $n=1,2,\ldots$ as follows
\begin{eqnarray*}
&& L_n \;=\; \inf \big\{ m: S_m > T_{n-1}+ \max \big( \sigma_{L_{n-1}+1},
\ldots, \sigma_{L_{n-1}+K_{n-1}} \big) \big\} , \\
&& K_n \;=\; L_n-L_{n-1}, \\
&& T_n \;=\; S_{L_n}.
\end{eqnarray*}
It is easy to see that $\{K_n\}$ is a Markov 
chain with state space $\{1,2,3,\ldots\}$  and transition probabilities given by
\[
P_{ij}:=\bP \left(K_{n+1}=j \mid K_n=i \right) = \bP \Big(S_{L_n+j-1} < S_{L_n}+ 
\max(\sigma_{L_n+1},\sigma_{L_n+2},\ldots,\sigma_{L_n+i})  \leq  S_{L_n+j} \Big)
\]
To simplify the notation, since $P_{ij} = \bP(K_1=j \mid K_0=i)$,
\begin{eqnarray} \nonumber
P_{ij} &=& \bP \left(S_{j-1} < \max(\sigma_1,\sigma_2,\ldots,\sigma_i) \leq  S_j \right) \\   \nonumber
\nonumber 
&=& \bP \left(\max(\sigma_1,\sigma_2,\ldots,\sigma_i) \leq  S_j \right) - 
\bP \left(\max(\sigma_1,\sigma_2,\ldots,\sigma_i) \leq  S_{j-1} \right) \\ \label{ab}
&=& \bE(1-e^{-\mu S_j})^i -  \bE(1-e^{-\mu S_{j-1}})^i \\  
&=& \sum_{k=0}^i \left(i \atop k \right) (-1)^k \bE \left[e^{-k\mu S_j} \right] \;-\;
\sum_{k=0}^i \left(i \atop k \right) (-1)^k \bE \left[e^{-k\mu S_{j-1}} \right] .   \label{f-Pij}
\end{eqnarray}
(The third equality stems from the fact that the $\sigma_i$'s are i.i.d.\ $\mbox{exp}(\mu)$.)
Since the arrivals are renewal, $S_j:=\tau_1+\cdots+\tau_j$ and 
$\bE[e^{-sS_j}] = \widehat{B}(s)^j$ where $\widehat{B}(s)
:=\int_0^\infty e^{-sx}dB(x)$ is the Laplace transform of the distribution $B$.
Thus, from (\ref{f-Pij}), 
\begin{eqnarray} \label{Pij-gated}
P_{ij} &=& \sum_{k=0}^i \left(i \atop k \right) (-1)^k \widehat{B}(k\mu)^{j-1} 
\left(\widehat{B}(k\mu)-1\right) .
\end{eqnarray}
The basic tool for establishing not only the positive recurrence of $\{K_n\}$ 
but also the finiteness of exponential moments for its stationary distribution
is the following generalization of the classic criterion of Foster which we 
state in the following Theorem. (This is Proposition 2.9 of Tweedie 
\cite[p. 829]{Tweedie}. See also Meyn and Tweedie
\cite{MT} for a comprehensive account of Foster--Lyapunov drift criteria 
in general state space Markov Chains.)
\begin{theorem}[Foster--Lyapunov Criterion]\label{th:Foster-Liapunov} Suppose 
	that a Markov Chain with countable state space, say $\bZ_+$, and transition
	probability matrix $P_{ij}$ is irreducible and let $N$ be a given natural
	number. If $V:\bZ_+\to \bR_+$, is a non-negative and $h:\bZ_+\to \bR_+$ a
	strictly positive function on the state space such that
	\begin{eqnarray}
		\sum_{j\in \bZ_+} P_{ij}V(j) &\leq  & d,\hspace{1.14in}i \leq  N,
		\label{precurrent1} \\
		\sum_{j\in \bZ_+} P_{ij}V(j) &\leq  & V(i)-h(i) ,\hspace{0.4in} i>N,  \label{precurrent2}
	\end{eqnarray}
	where $d>0$, then the Markov Chain is positive recurrent and its (unique
	due to its irreducibility) stationary distribution, $\pi$, satisfies $\sum_j\pi(j) h(j) <\infty$. 
	Furthermore we have 
	$\sum_{j}P_{ij}^n h(j) \rightarrow \sum_j \pi_j h(j)$ as $n \rightarrow \infty$.
\end{theorem}
We now use the above to prove the following
\begin{proposition} \label{prop:expphi} Suppose that the interarrival times are positive 
with probability 1. Then the Markov chain $\{K_n\}$ defined above is positive recurrent 
and its invariant distribution $\{\pi_i\}$ satisfies the equations
\begin{eqnarray} \label{invariant-GM}
\pi_i &=& \sum_{j=1}^\infty \pi_j P_{ji} \hspace{0.3in} i=1,2,3,\ldots \\ 
\nonumber
1  &=&\sum_{i=1}^\infty \pi_i .
\end{eqnarray}
Furthermore, for $\theta \in (1, \frac{1}{\widehat{B}(\mu)})$ this invariant distribution
$\{\pi_i\}$ satisfies 
\begin{equation} \label{expmoments}
\sum_{i=1}^\infty \pi_i \theta^i < \infty.
\end{equation}
\end{proposition}
\begin{proof}
Set $x:= \left(1-e^{-\mu S_j} \right)$, $y:= \left(1-e^{-\mu S_{j-1}} \right)$. Taking into
account the identity
\begin{equation}    \label{aibi}
x^i -y^i \;=\; (x-y) (x^{i-1}+ x^{i-2}y+ \cdots + y^{i-1})
\end{equation}
and the inequality $0 <y < x < 1$, implying that the sum in the parenthesis
in (\ref{aibi}) is bounded above by $i$,  we have from (\ref{ab}) 
\begin{eqnarray}  \nonumber
P_{ij} \;=\; \bE\left[ \left(1-e^{-\mu S_{j-1}}\right)^i - \left(1-e^{-\mu S_{j}}\right)^i \right]
& \leq & \bE \left[ i \left(e^{-\mu S_{j-1}} -e^{-\mu S_j}\right) \right]  \\   \label{sjsj1}
&=& i \widehat{B}(\mu)^{j-1} \left( 1-\widehat{B}(\mu)\right) .
\end{eqnarray} 
We will pick $\theta \in (1, \frac{1}{\widehat{B}(\mu)})$, and apply Theorem \ref{th:Foster-Liapunov}
with $h(i):=\theta^i$, $V(i)=L \theta^i$, where $L>0$. Then, using (\ref{sjsj1}),
\[
\sum_{j=1}^\infty P_{ij}V(j) \; \leq \;  L \sum_{j=1}^\infty  i \widehat{B}(\mu)^{j-1} 
\left( 1-\widehat{B}(\mu)\right) \, \theta^j \;=\; 
i\, L\theta \frac{1- \widehat{B}(\mu)}{1-\theta \widehat{B}(\mu)} .
\]
Define $N :=\sup\left\{i \geq 1 : i\, L\theta \frac{1- \widehat{B}(\mu)}{1-\theta \widehat{B}(\mu)} 
\geq  (L-1) \theta^i \right\}$ which is finite since $\theta >1$. With this choice of $N$ and 
with $d$ chosen so that
$d \geq N  \, L\theta \frac{1- \widehat{B}(\mu)}{1-\theta \widehat{B}(\mu)}$ the conditions
(\ref{precurrent1}) and (\ref{precurrent2})  of Theorem \ref{th:Foster-Liapunov} are satisfied.
Therefore, 
the Markov chain $\{K_n\}$ is positive recurrent with invariant distribution given by the 
solution of (\ref{invariant-GM}) and which satisfies (\ref{expmoments}).
\end{proof}
The global balance equations (\ref{invariant-GM}) for the invariant distribution $\{\pi_i\}$ are
not easy to solve. In this section we will present an alternative approach based on the 
factorial moments of the invariant distribution, applicable under a light traffic 
condition.

\subsection{Analysis in the light traffic case}

The analysis in this section is carried out under the {\em light traffic condition}
\begin{equation} \label{light_traffic_GM}
b_1:=\widehat{B}(\mu) \,<\, \frac{1}{2}.
\end{equation}
Note that when arrivals are Poisson, i.e.\ $\widehat{B}(s) = \frac{\lambda}{\lambda+s}$, the light traffic condition is
equivalent to $\rho:= \lambda/\mu <1$. Let
\[
\phi(z):= \sum_{i=1}^\infty z^i \pi_i
\]
denote the probability generating function of the number of customers, $K_0$, served in
a service stage in stationarity. (Note that $\bP(K_0 >0)=1$.) 
\begin{lemma} \label{le:pgf}
Under the light traffic condition (\ref{light_traffic_GM}) the probability generating 
function $\phi(z)$ converges on the disk $\{z: |z| \leq \theta\}$ where 
$\theta \in (2,\widehat{B}(\mu)^{-1})$. The descending factorial moments are 
all finite, given by 
\begin{equation}  \label{descending_factorial}
\sum_{n=1}^\infty n(n-1)(n-2)\cdots(n-m+1) \pi_n  \;=\; \phi^{(m)}(1), 
\quad \mbox{ for } m=0,1,2,\ldots ,
\end{equation}
and they satisfy the conditions
\begin{eqnarray}  \label{phi-pm}
&& \sum_{m=0}^\infty (-1)^m \frac{\phi^{(m)}(1)}{m!} \;=\; 0, \\ \label{phi-pp}
&& \sum_{m=0}^\infty  \frac{\phi^{(m)}(1)}{m!} \;=\; \phi(2) < \infty.
\end{eqnarray} 
\end{lemma}
\begin{proof}
From Proposition \ref{prop:expphi} we see that, under the light traffic condition, 
with $\theta$ as in the statement of Lemma \ref{le:pgf}, $\bE[\theta^{K_0}]<\infty$ 
which shows that $\phi$ converges on the disk $\{z: |z| \leq \theta\}$.
From this, the finiteness of the 
descending factorial moments and (\ref{descending_factorial}) are immediate.
Also, the Taylor expansion of $\phi$ around 0,  
$\phi(z)=\sum_{m=0}^\infty \frac{z^m}{m!}\phi^{(m)}(0)$,
has radius of convergence $R>2$. Hence the function $\phi$ is holomorphic
in the disk $\{z: |z|< R\}$ and therefore the Taylor series of $\phi$ around $z=1$ has
radius of convergence greater than 1 and in the disk $\{z: |z-1| < R-1\}$ 
\begin{equation} \label{power_series_1}
\phi(z)=\sum_{m=0}^\infty \frac{(z-1)^m}{m!}\phi^{(m)}(1).
\end{equation}
Thus we may set $z=0$ in (\ref{power_series_1}) and since $\phi(0)=0$, we 
obtain (\ref{phi-pm}). 
Similarly, set $z=2$ in (\ref{power_series_1}) which is in the disk of 
convergence around $z=1$. Therefore $\phi(2)<\infty$ and (\ref{phi-pp}) follows.
\end{proof}
We next obtain an infinite system satisfied by the {\em binomial moments}
$\bE {{K_0} \choose m} = \frac{1}{m!}\phi^{(m)}(1)$, $m=1,2,\ldots$, 
and use the solution to express the stationary distribution $\{\pi_i\}$ in terms of them.
\begin{theorem}  \label{th:MGs1}
Suppose the interarrival time distribution satisfies the light traffic 
condition (\ref{light_traffic_GM}) and define the quantities
\begin{equation} \label{def_xm}
x_m:=\frac{\phi^{(m)}(1)}{m!},\hspace{0.1in}b_k:=\widehat B(k\mu),\hspace{0.1in} 
\mbox{ and } \hspace{0.1in} a_{mk}:= 
\frac{b_k^{m-1}}{(1-b_k)^m}, 
\hspace{0.1in} k,m=1,2,\ldots.
\end{equation}
Then the sequence of  binomial moments $\{x_m\}$, $m=1,2,\ldots$, satisfies the 
infinite linear system
\begin{alignat}{2}   \nonumber
x_1 &=  1 +  && \sum_{k=1}^\infty x_k (-1)^{k-1} (a_{1k} -1), \\  \label{xm} 
x_m &=         && \sum_{k=1}^\infty x_k (-1)^{k-1}a_{mk}, \hspace{0.3in} m=2,3,\ldots
\end{alignat}
\end{theorem}
\begin{proof}
Multiplying (\ref{Pij-gated}) by $z^j$ and summing over $j$ we obtain
\begin{eqnarray} \nonumber
\bE[z^{K_1} \mid K_0=i] &=&  \sum_{j=1}^\infty z^j \sum_{k=1}^\infty 
\left( i \atop k \right) (-1)^k b_k^{j-1}(b_k-1) \\
& = & \sum_{k=1}^\infty \left( i \atop k \right) (-1)^k (b_k-1)
\sum_{j=1}^\infty z^j b_k^{j-1}.   \label{MC}
\end{eqnarray}
Since the geometric series in the last sum converges (at least for $|z|< \theta$) 
(\ref{MC}) becomes
\begin{equation} \label{Exp-Fubini}
\bE[z^{K_1}\mid K_0]=\sum_{k=1}^{\infty }  \left(K_0 \atop k \right)
(-1)^k \frac{b_k-1}{1-b_kz}z \,.
\end{equation}
Noting that $b_k \leq b_1$ for $k=1,2,\ldots$, we can easily show that 
\[
\left| (-1)^k \frac{b_k-1}{1-b_kz}z \right| 
\leq \frac{|z|}{1- b_1|z|}.
\]
Also, from Lemma \ref{le:pgf} we see that $\bE[2^{K_0} ] <\infty$. Hence
\begin{eqnarray*}
\sum_{k=1}^{\infty } \mathbb{E} \left| \left(K_0 \atop k \right)
(-1)^k \frac{b_k-1}{1-b_kz}z \right|
& \leq & \sum_{k=1}^{\infty } \mathbb{E} \left[ \left(K_0 \atop k \right) \right]
\, \frac{|z|}{1- b_1|z|} \;=\;  \frac{|z|}{1- b_1|z|}  \mathbb{E} \left[ \sum_{k=1}^{\infty} 
\left(K_0 \atop k \right) \right]  \\
& = & \frac{|z|}{1- b_1|z|}  \mathbb{E} \left[  \sum_{k=1}^{K_0} \left(K_0 \atop k \right) 
\right] \;=\; 
\frac{|z|}{1- b_1|z|}  \mathbb{E} \left[ 2^{K_0} -1\right] < \infty,
\end{eqnarray*}
the interchange of infinite summation and expectation above being justified by
the positivity of the terms. 
Take expectation with respect to $K_0$ in (\ref{Exp-Fubini}) and use the above to justify
the interchange of summation and expectation. Since by assumption $K_0$ has
the invariant distribution we obtain the following expression for the generating function 
of the number of customers served in each stage:
\begin{eqnarray*} 
\bE[z^{K_0}] &=&  \sum_{k=1}^\infty \bE[K_0(K_0-1)\cdots(K_0-k+1)]
\frac{(-1)^{k-1}}{k!}\frac{1-b_k}{1-b_kz}z.
\end{eqnarray*}
Using also (\ref{descending_factorial}) 
we can rewrite the above as 
\begin{eqnarray}  \label{Generating function}
\phi(z) \;=\; \sum_{k=1}^\infty\phi^{(k)}(1)\frac{(-1)^{k-1}}{k!}
 \frac{z(1-b_k)}{1-b_kz}.
\end{eqnarray}
This last expression represents the probability generating function of 
the steady-state number of customers served in a stage as an infinite linear combination 
of p.g.f.'s of geometric random variables, the $k$th of which has probability of success
$1-b_k$. Of course (\ref{Generating function}) involves the
factorial moments of the unknown distribution on the right hand side, so the determination of
$\phi$ from this is not straight forward. Note that, if 
\begin{equation} \label{akbk}
f_k(z):=\frac{z(1-b_k)}{1-b_k z}.
\end{equation}
then using Leibniz' rule
\begin{equation} \label{Dmfk}
f^{(m)}_k(1) = \frac{m!b_k^{m-1}}{(1-b_k)^{m}}
\end{equation}
and thus differentiating
(\ref{Generating function}) $m$ times with respect to $z$ term by term 
and evaluating at $z=1$ gives
\begin{equation} \label{Generating function2}
\phi^{(m)}(1)=\sum_{k=1}^\infty \phi^{(k)}(1)\frac{(-1)^{k-1}}{k!}\; 
 \frac{m!b_k^{m-1}}{(1-b_k)^{m}}
\hspace{0.3in} m=1,2,\ldots
\end{equation}
The argument that justifies the interchange of the infinite sum 
and the $m$th derivative operator is given in Proposition (\ref{prop:Dmphi}) of the Appendix.
Dividing both sides of the above equation by $m!$ and using the definitions 
(\ref{def_xm}) we obtain the following system for the binomial moments.
\begin{equation}  \label{xm0}
x_m  \;=\; \sum_{k=1}^\infty x_k (-1)^{k-1}a_{mk}, \hspace{0.3in} m=1,2,3,\ldots .
\end{equation}
This agrees with the system (\ref{xm}) except for the first equation corresponding to $m=1$. 
In fact (\ref{xm0}) is homogeneous and therefore would at best determine the unknown 
quantities $x_m$ up to a multiplicative constant. We may however appeal to Lemma \ref{le:pgf} 
and in particular (\ref{phi-pm}) which we may write in the form
\[
0 \;=\; 1 - \sum_{k=1}^\infty (-1)^{k-1} x_k.
\]
Adding it to the first equation of (\ref{xm0}) gives the first equation of (\ref{xm}) and thus 
a non-homogeneous system.
\end{proof}

We are now ready to express the invariant distribution (\ref{invariant-GM}) in terms of the 
factorial moments. Assuming that the sequence 
$\{x_m\}$ has been determined note from (\ref{Generating function}) that
\begin{eqnarray} \nonumber
\sum_{i=1}^\infty \pi_i z^i &=& \sum_{k=1}^\infty x_k (-1)^{k-1} \,
\frac{z(1-b_k)}{1-b_kz} 
= \sum_{k=1}^\infty x_k (-1)^{k-1} \sum_{i=1}^\infty z^i 
\left( 1-b_k \right) b_k^{i-1} \\ \nonumber
&=&  \sum_{i=1}^\infty z^i \sum_{k=1}^\infty x_k (-1)^{k-1}  
\left( 1-b_k \right) b_k^{i-1} .
\end{eqnarray}
(The interchange of the two sums is justified by the fact that they converge absolutely 
as a result of (\ref{phi-pp}).) Thus we have the following expression for the stationary distribution 
\begin{equation} \label{pi-final}
\pi_i \;=\; \sum_{k=1}^\infty x_k (-1)^{k-1}  \left( 1-b_k \right) 
b_k^{i-1}, \hspace{0.2in} i=1,2,\ldots.
\end{equation}
Clearly, from (\ref{def_xm}), the quantities $x_k$ are positive and hence 
(\ref{pi-final}) gives the stationary distribution as an alternating sum of geometric 
probabilities.


In particular, if the arrival process is Poisson ($\lambda$) and if we set 
$\rho:=\lambda/\mu$ then $\widehat{B}(\mu) = 
\frac{\lambda}{\lambda +\mu} = \frac{\rho}{1+\rho}$ and thus, for a
gated $M/M/\infty$ system, the {\em pertinent light traffic condition is $\rho<1$.}

Using the definitions (\ref{def_xm}) we rewrite the system (\ref{xm}) as
\begin{alignat}{2}   \nonumber
x_1 &= 1 +  &&\sum_{k=1}^\infty x_k (-1)^{k-1} \frac{b_k}{1-b_k} \\  \label{xm_system} 
x_m &=       &&\sum_{k=1}^\infty x_k (-1)^{k-1} \frac{b_k^{m-1}}{(1-b_k)^{m}}, 
\hspace{0.2in} m=2,3,\ldots. 
\end{alignat}
Theorem \ref{th:MGs1} implies that the above system has a positive solution. Taking
into account (\ref{phi-pp}) we see that 
\[
\sum_{m=1}^\infty x_m  = \phi(2)-1<\infty
\] 
i.e.\ $x \in l_1:=\big\{(x_1,x_2,\ldots): \sum_{i=1}^\infty |x_i|< \infty \big\}$. 
(Note that in this section the numbering of the elements of the sequences 
begins with 1 unlike Section \ref{sec:MG1}. For this reason we use a 
different symbol to designate this space.) We denote by 
$\Vert x \Vert :=\sum_{i=1}^\infty |x_i|$ the $l_1$ norm of $x$. For any bounded,
linear operator $L: l_1 \rightarrow l_1$, $\Vert L \Vert $ denotes the induced operator
norm. We next examine the system (\ref{xm_system}) from the point of view of 
linear operators in the linear space $l_1$ in order to show that the solution $x$ 
whose existence is asserted by Theorem \ref{th:MGs1} is unique. Define for simplicity 
\begin{equation} \label{r-def}
  q_k:=\frac{b_k}{1-b_k},\qquad  r:=q_1<1,
\end{equation}
the last inequality following from the fact that $b_1 <1/2$. Also
$ 0<q_{k+1}\le q_k\le r<1$, $k=1,2,\ldots,$  and $\lim_{k\rightarrow \infty} q_k=0$.
The system (\ref{xm_system}) may be written in the form
\begin{equation} \label{GMsys}
x = e_1 + A x
\end{equation}
where $e_1 = (1,0,0,\ldots)$ and $A$ is the infinite matrix with entries
\begin{equation}  \label{Amk}
 A_{mk}=
 \begin{cases}
  (-1)^{k-1}q_k,&m=1,\\
  (-1)^{k-1}(q_k^{m-1}+q_k^m),&m\ge2.
 \end{cases}
\end{equation}
Let $S$ be the right shift on $l_1$, defined by $ (Sy)_1=0$ and $(Sy)_m=y_{m-1}$
for $m=2,3,\ldots$, and put $T=I+S$. Define $C:l_1\to l_1$ by
\begin{equation}   \label{Cz}
 (Cz)_m=\sum_{k=1}^{\infty}(-1)^{k-1}q_k^m z_k,\qquad m\ge1, \;\;\; z \in l_1.
\end{equation}
Then $A$ can be written as the product $A=TC$ where both $T$ and $C$ are bounded 
operators on $l_1$. Indeed,  $\Vert T \Vert \le 2$,  
\[
\Vert C \Vert \;=\;  \sum_{m=1}^\infty |C_{mk}|  \;=\;  \frac{q_k}{1-q_k} 
\; \le \;\frac{r}{1-r},
\]
and hence $A$ is also a bounded linear operator in $l_1$.
Our intention is to show that the operator $I-A$ is invertible in $l_1$. While this
would be true if $\Vert A \Vert$ were less than 1, this is not nessasarily true.
In the following lemma we will show however that $\Vert A^n \Vert$  converges
to zero geometrically fast as $n \rightarrow \infty$ and thus obtain the invertibility
of $I-A$. 

\begin{lemma}  \label{le:norm}
For every integer $n\ge1$,
\begin{equation} \label{GM-An}
A^n=(TC)^n=T(CT)^{n-1}C=TG^{n-1}C, 
\end{equation}
where $G:= CT$ is a bounded linear operator on $l_1$ and
\begin{equation} \label{norm-bound}
  \Vert A^n\Vert \le \frac{2r^n}{1-r}.
\end{equation}
\end{lemma}
\begin{proof}
For any $n\geq 1$ $A^n = (TC)^n = T (CT)^{n-1} C$ and (\ref{GM-An}) follows from the 
definition of $G$. 
Note that the operator $G$ satisfies
 \begin{equation}\label{eq:K}
  (Gz)_m  = q_1^m z_1 + \sum_{k=2}^\infty (-1)^{k-1} q_k^m \big(z_k+ z_{k-1} \big) 
  =\sum_{k=1}^{\infty}    (-1)^{k-1}(q_k^m-q_{k+1}^m)z_k,  \qquad z\in l_1.
 \end{equation}
The second equality holds since the series inolved converge absolutely.
For every $m\ge1$, the sum of the absolute value of the coefficients, 
$\left|(-1)^{k-1}(q_k^m-q_{k+1}^m)\right|$, satisfy
 \begin{equation}\label{eq:telescope}
  \sum_{k=1}^{\infty}(q_k^m-q_{k+1}^m)=q_1^m=r^m.
  \end{equation}
Fix $z\in l_1$ and consider the sequence of $l_1$ elements defined by
\[
 z^{(n)}:=G^nCz,\qquad n=0,1,2,\ldots.
\]
(The fact that these are elements of $l_1$ is a consequence of the fact that 
$C$ and $G$ are bounded operators on that space.) We prove by induction in $j$ that
\begin{equation}\label{eq:coordinate-bound}
 \big| z_m^{(n)}\big| \le r^{m+n} \Vert z \Vert,  \quad \forall m, \quad n=0,1,2,\ldots .
\end{equation}
For $n=0$ (\ref{eq:coordinate-bound}) holds because, by the definition of $C$ (\ref{Cz}),
\[
 \big| (Cz)_m\big|  \le\sum_{k=1}^{\infty}q_k^m\big| z_k\big|  \le r^m \Vert z \Vert.
\]
Suppose \eqref{eq:coordinate-bound} holds for $n$. Using
\eqref{eq:K}, the nonnegativity of $q_k^m-q_{k+1}^m$, and the fact that
$r^{k+n}\le r^{n+1}$ for $k\ge1$, we obtain
\begin{align*}
\big| z_m^{(n+1)}\big| &= \big| (G z^{(n)})_m \big|
 \; \le \;
 \sum_{k=1}^{\infty}(q_k^m-q_{k+1}^m) \big| z_k^{(n)}\big| \; \le \; 
  \Vert z\Vert \sum_{k=1}^{\infty}(q_k^m-q_{k+1}^m)r^{k+n} \\
 &\le  r^{n+1} \Vert z\Vert
 \sum_{k=1}^{\infty}(q_k^m-q_{k+1}^m)\; = \; r^{m+n+1} \Vert z\Vert,
\end{align*}
the last equality following from the telescopic series \eqref{eq:telescope}.
This proves \eqref{eq:coordinate-bound}.

Thus, from the above, for any $z \in l_1$, $A^n z = TG^{n-1}Cz $ and hence,
from (\ref{eq:coordinate-bound}),
\[
 \Vert A^n z \Vert  \le 2 \Vert G^{n-1}Cz \Vert \le 2 \Vert z^{(n-1)}\Vert
= 2 \sum_{m=1}^\infty |z^{(n-1)}_m| \leq 2 
\sum_{m=1}^\infty  r^{m+n-1} \Vert z \Vert   =\frac{2r^n}{1-r} \Vert z\Vert.
\]
Since $z$ is an arbitrary element of $l_1$ we obtain (\ref{norm-bound}).
\end{proof}
\begin{theorem}  \label{th:GMsystem}
The operator $I-A$ on $l_1$ has a bounded inverse and the system (\ref{GMsys})
has a unique solution $x \in l_1$. The solution is positive: $x_i >0$ for $i=1,2,\ldots$.
\end{theorem}
\begin{proof}
Since $l_1$ is a Banach space the convergence of the series $\sum_{n=0}^\infty \Vert A^n \Vert$
implies the convergence of the Neumann series $B:=\sum_{n=0}^\infty A^n$. This in turn implies 
that $(I-A)^{-1}$ exists and is equal to this series.
From Lemma \ref{le:norm}
\[
\Vert B \Vert \;=\; 
\Vert \sum_{n=0}^\infty A^n \Vert  \; \leq \; \sum_{n=0}^\infty \Vert A^n \Vert 
\; \leq \;  1+  \sum_{n=1}^\infty \frac{2r^n}{1-r} \;=\; \frac{1+r^2}{(1-r)^2}.
\]
In particular the system (\ref{GMsys}) has a unique solution
\[
x = \sum_{n=0}^\infty A^n e_1.
\]
The poisitivity of the solution is immediate since 
$x_i = \frac{1}{i!}\sum_{n=1}^\infty \pi_n n(n-1)\cdots (n-i+1)$.
\end{proof}

\begin{remark}
When the arrival process is Poisson $(\lambda)$ the system (\ref{xm_system}) becomes
\begin{alignat}{2}   \nonumber
x_1 &= 1 +  &&\sum_{j=1}^\infty x_j (-1)^{j-1} \frac{\rho}{j} \; , \\  \label{xm_system_pois} 
x_i  &=        &&\sum_{j=1}^\infty x_j (-1)^{j-1} \Big( 1+ \frac{\rho}{j} \Big) \, \Big( \frac{\rho}{j}\Big)^{i-1}, 
\hspace{0.2in} i=2,3,\ldots
\end{alignat}
where $\rho:=\lambda/\mu$. The probability that the stationary service stage serves $i$ customers 
(\ref{pi-final}) is then
\begin{equation}   \label{pi-final-pois}
\pi_i \;=\; \sum_{j=1}^\infty x_j (-1)^{j-1} \frac{j \rho^{i-1}}{(\rho+j)^i}, \;\;\; i=1,2,\ldots.
\end{equation}
Figure \ref{fig6} shows how an approximation based on a finite section of this system compares with simulation
results when $\rho<1$.
\end{remark}

Finally we examine the problem of approximating the solution of (\ref{GMsys}) by a finite section
and we show that the truncation error converges to zero with rate  $O(r^N)$  where $r$ is 
defined in (\ref{r-def}). For $N\ge1$, put $ A_N=(a_{mk})_{1\le m,k\le N}$,  
$e_1^{(N)}=(1,0,\ldots,0)\in\mathbb R^N$. For each $N$ the finite section of the system 
(\ref{GMsys}) is 
\begin{equation}\label{eq:finite-system}
 (I_N-A_N)x^{(N)}=e_1^{(N)}.
\end{equation}
Write $\widetilde x^{(N)}=(x_1^{(N)},\ldots,x_N^{(N)},0,0,\ldots)$
for its extension to an infinite sequence.

\begin{figure}[h!]
  \begin{center}
    \includegraphics[width=6.2in]{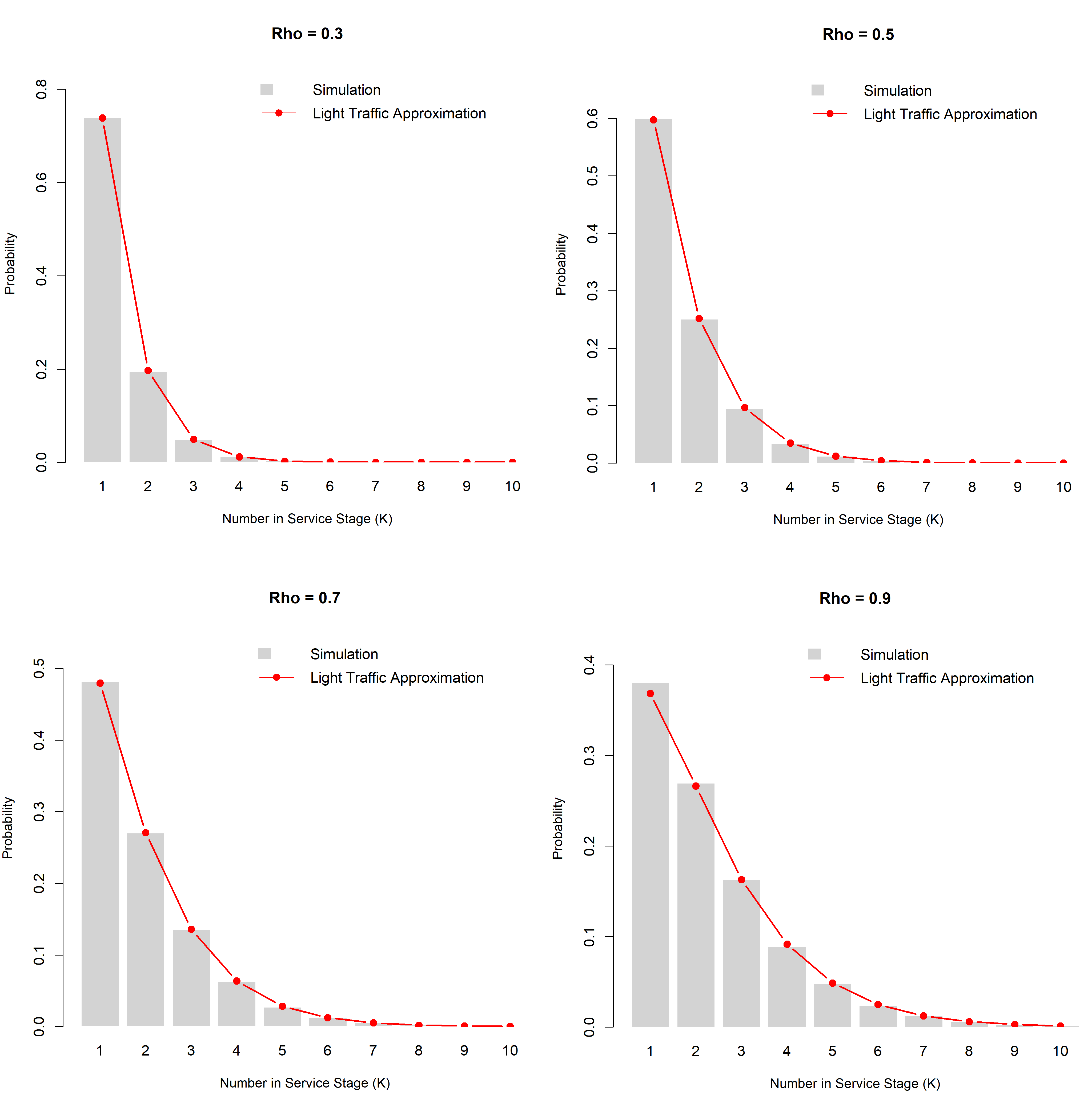}
  \end{center}
  \caption{The light traffic approximation for the gated $GI/M/\infty$ system with Poisson ($\lambda$)
  arrivals. $\rho:=\lambda/\mu$ and the approximation uses $N=50$ terms. The simulation uses $10^5$
  service stages and a short runin (1000) periods. The system (\ref{xm}) is truncated at $N=50$ and 50 
  terms are taken in the series (\ref{pi-final}).}
  \label{fig6}
\end{figure}

\begin{theorem}\label{thm:main}
In the above framework,  for every $N\ge1$, the matrix $I_N-A_N$ is nonsingular. 
Furthermore, the sequence of solutions $\{\widetilde{x}^{(N)}\}$ 
of \eqref{eq:finite-system} satisfy 
$\Vert \widetilde x^{(N)}-x\Vert \to 0$. More precisely, for every $N\ge1$,
\begin{equation}\label{eq:main-error}
\Vert \widetilde x^{(N)}-x\Vert  \le  \left(1+M\,\frac{2q_N}{1-q_N}\right)
 \frac{(1+r)r^N}{(1-r)^2}
\end{equation}
where $M:=\frac{1+r^2}{(1-r)^2}$.
\end{theorem}
\begin{proof}
Let $S_N$ be the right shift in $\mathbb R^N$ with its last coordinate discarded, i.e.\
for $z\in \bR^N$ $(S_N z)_1 =0$ and $(S_N z)_i =z_{i-1}$ for $2 \leq i \leq N-1$.
Set \( T_N=I_N+S_N\), $C_N=\big((-1)^{k-1} q_k^m\big)_{1\le m,k\le N}$.
The same calculation as before gives
\begin{equation}\label{eq:finite-factor} \nonumber
 A_N=T_NC_N.
\end{equation}
However, the reversed product $G_N=C_NT_N$ has a special last column:
\begin{equation}\label{eq:finite-K} \nonumber
 (G_N)_{mk}=  \begin{cases}   (-1)^{k-1}(q_k^m-q_{k+1}^m),&1\le k<N,\\
  (-1)^{N-1} q_N^m,&k=N.
 \end{cases}
\end{equation}
Then for every $n$ $A_N^n = T_N G_N^{n-1} C_N$ and using the argument in 
the proof of Lemma \ref{le:norm} we can prove that the following bound holds 
for the norm of $A_N$
\begin{equation} \label{normANn}  \nonumber
\Vert A_N^n \Vert \leq \frac{2r^n}{1-r}  .
\end{equation}
Therefore $\Vert \sum_{n=0}^\infty A_N^n \Vert 
\leq  1 + \sum_{n=1}^\infty \Vert A_N^n \Vert \leq  \frac{1+r^2}{(1-r)^2}=: M$ which
implies that $I_N - A_N$ is invertible with 
\begin{equation} \label{Norm-Inv}
\Vert (I_N - A_N)^{-1} \Vert  \; \leq \; M.
\end{equation} 
The above bound holds for all $N$. Therefore, for all $N$ the system 
(\ref{eq:finite-system}) has a unique solution.

We next give an estimate for the truncation error. Set $y=Cx$. Since $x=e_1+Ty$, we have
\begin{equation}\label{eq:y-equation}
 y=Ce_1+Gy,\qquad (Ce_1)_m=r^m.
\end{equation}
Let $\Vert x \Vert_{\infty} := \sup_{i\in \bN} |x_i|$. From (\ref{eq:K}) 
$|(Gy)_m| \leq \sum_{k=1}^\infty (q_k^m -q_{k+1}^m) |y_k| \leq \Vert y \Vert_\infty r^m$ 
and therefore $\Vert G y \Vert _\infty \leq r  \Vert y\Vert_\infty$.
Taking the supremum norm in (\ref{eq:y-equation}) we have 
$ \Vert y\Vert_\infty \le r+r \Vert y\Vert_\infty$, so 
\[
  \Vert y\Vert_\infty\le\frac r{1-r}.
\]
From (\ref{eq:K}) we obtain 
\begin{equation}\label{eq:y-decay} \nonumber
|y_m|\le r^m+r^m\Vert y\Vert_\infty  \le\frac{r^m}{1-r},\qquad m\ge1.
\end{equation}
For $m\ge2$, $x_m=y_m+y_{m-1}$, and hence
\begin{equation}\label{eq:x-decay} \nonumber
|x_m| \le \frac{(1+r)r^{m-1}}{1-r}.
\end{equation}
It follows that, for every $N\ge1$,
\begin{equation}\label{eq:tail}
 t_N:=\sum_{m=N+1}^\infty |x_m|  \le\frac{(1+r)r^N}{(1-r)^2}.
\end{equation}

Let $R_N: l_1\to\mathbb R^N$ denote the linear operator that retains the 
first $N$ coordinates and $E_N:\mathbb R^N\to l_1$ the linear operator
that extends an element in $\mathbb R^N$ by adding zeros to an element of
$l_1$, and put $P_N:=E_NR_N$ and $Q_N:=I-P_N$. Restricting the infinite system to 
the first $N$ equations and variables gives
\[
 (I_N-A_N)R_Nx  = e_1^{(N)}+R_NA Q_Nx.
\]
Subtracting \eqref{eq:finite-system} therefore gives the following expression
for the truncation error:
\begin{equation}\label{eq:residual}
 (I_N-A_N)(R_Nx-x^{(N)})=R_NA Q_Nx.
\end{equation}
Note that $R_NA Q_Nx = \sum_{m=1}^N\sum_{k=N+1}^\infty a_{mk} x_k$ expresses
the effect of omitting  the elements $x_{N+1}$, $x_{N+2},\ldots$ on the first $N$ elements
of $x$. From (\ref{Amk}) we have the bound for the column sum $\sum_{m=1}^N |a_{mk}| \leq
\frac{2q_N}{1-q_N}$. Using also the monotonicity of $q_k$, and \eqref{eq:tail},
\begin{equation} 
 \Vert R_NA Q_Nx\Vert
 \le
 \sum_{k>N}|x_k| \sum_{m=1}^{N} |a_{mk}|  
 \le
 \frac{2q_N}{1-q_N}\,
 \Vert Q_Nx \Vert  =\frac{2q_N}{1-q_N}\,t_N.
 \label{eq:residual-bound}
\end{equation} 
Combining this with the uniform inverse bound
(\ref{Norm-Inv}) yields
\[
 \Vert R_Nx-x^{(N)}\Vert  \le M\,\frac{2q_N}{1-q_N}\,t_N.
\]
The first $N$ coordinates and the remaining coordinates have disjoint
supports. Thus
\[
 \Vert \widetilde x^{(N)}-x\Vert
=\Vert x^{(N)}-R_Nx\Vert+\Vert Q_Nx\Vert\\
\le  \left(1+M\,\frac{2q_N}{1-q_N}\right)t_N.
\]
Using \eqref{eq:tail} proves \eqref{eq:main-error}.
The rate with which the error goes to zero is   $O(r^N)$ because of (\ref{eq:tail}).  
\end{proof}
In particular, when the arrival process is Poisson ($\lambda$), 
$\widehat{B}(s) = \frac{\lambda}{\lambda+s}$, and $b_k= \widehat{B}(k\mu) 
= \frac{\lambda}{\lambda+k\mu} = \frac{\rho}{\rho + k}$ with $\rho := \lambda/\mu<1$.
Then $q_k = \rho/k$, $r=q_1=\rho$, and the system (\ref{xm_system}) is written as 
\begin{alignat}{2}   \nonumber  
x_1 &= 1  + &&\sum_{k=1}^\infty x_k (-1)^{k-1} \frac{\rho}{k} \\
x_m &=&&  \sum_{k=1}^\infty x_k (-1)^{k-1} \left(1+\frac{\rho}{k}\right) \left(\frac{\rho}{k}\right)^{m-1},  
\hspace{0.3in} m=2,3,\ldots.  \label{Poisson_system}
\end{alignat}
The truncation error is bounded by
\begin{equation} \label{error_bound}
\Vert \widetilde x^{(N)}-x\Vert  \le  \left(1+\frac{1+\rho^2}{(1-\rho)^2} \,\frac{2\rho}{N+1-\rho}\right)
 \frac{(1+\rho)\rho^N}{(1-\rho)^2}.
\end{equation}


\section{Appendix} \label{sec:Appendix}

\subsection{Markov chains on general state spaces} \label{sec:MCGeneralState}


Here we mention some basic concepts and definitions pertaining to 
the analysis of Markov chains in general state spaces. We refer the
reader to \cite{Tweedie} for an introduction and to \cite{MT} 
and \cite{Douc} for a detailed account of the theory.
Suppose that $\boldsymbol{\Phi}:=\{\Phi_n\}$, $n\in \bN$ is a Markov 
chain with state space ${\sf X}$ (which is a general set)  
equipped with a family of $\sigma$-fields of subsets of ${\sf X}$, $\cB({\sf X})$, 
and transition kernel $P(\cdot,\cdot):{\sf X}\times \cB({\sf X}) \rightarrow [0,1]$. 
For any $A \in \cB({\sf X}) $ let $\tau_A := \inf\{n \geq 1: \Phi_n \in A\}$ and 
$L(x,A) :=\bP_x(\tau_A < \infty)$. 

If there exists a measure $\phi$ on 
$\big({\sf X}, \cB({\sf X})\big)$ satisfying the condition 
$\phi(A) >0 \; \Rightarrow \; L(x,A)>0$  for all $A \in \cB({\sf X})$ and $x \in {\sf X}$ 
then $\phi$ is called an {\em irreducibility measure} for the chain. It turns
out that if an irreducibility measure exists for a chain $\{\Phi_n\}$ then
an {\em maximal irreducibility measure $\psi$} exists for this chain.
$\boldsymbol{\Phi}$ is then called $\psi$-irreducible, the measure $\psi$ is essentially unique, 
and it satisfies the following two conditions.
\begin{itemize}
	\item[(i)] $\psi(A)>0$ implies that $L(x,A)$ for every $x\in {\sf X}$ and $A \in \cB({\sf X})$
	\item[(ii)] If $\psi(A)=0$ then $\psi(\overline{A})=0$ where 
	$\overline{A}=\{x: L(x,A)>0\}$. Note that $\overline{A}$ is the set of points in ${\sf X}$ 
	from which the chain can reach $A$ with positive probability.
\end{itemize}

A set $C$ in $\cB({\sf X})$ is called {\em small}
if there exists a constant $\delta >0$ and a probability measure $\phi$  such that,
\begin{equation} \label{def:small}
\mbox{\em there exists $m\geq 1$ such that for every $x \in C$  and every 
$B \in \cB({\sf X})$, $P^m(x,B) \geq \delta \phi(B)$.}
\end{equation}
If there exists a small set which satisfies the above definition with $m=1$ then 
the chain is {\em strongly aperiodic} (\cite[p.114]{MT}).
Let $\cB^+ := \{ A \in  \cB({\sf X}) : \psi(A) >0\}$. 
When a small set $C$ is in $\cB^+$ then the measure $\phi$ in Definition (\ref{def:small})
is an irreducibility measure. Finally, a set $C \in \cB({\sf X})$ is {\em accessible} 
if $L(x,C) >0$ for all $x \in {\sf X}$. 

The following definition generalizes the concept of irreducibility from
Markov chains on a countable state space to chains on a general state space.
\begin{definition}\label{def:irreducible} A Markov kernel is $\psi$-irreducible if it 
admits an accessible small set.
\end{definition} 
This definition describes irreducibility in the general framework. For any $x$ in
the state space ${\sf X}$ and any set $B \in \cB$ of importance, in the sense that
$\phi(B)>0$, there exists $n \in \bN$ such that $\bP(\Phi_n\in B | \Phi_0=x)>0$.
Indeed, since $C$ is accessible, for given $x$ there exists $n_1$ such that 
$\bP(\Phi_{n_1}\in C | \Phi_0=x) >0$ and 
\begin{eqnarray*}
\bP(\Phi_{n_1+m}\in B | \Phi_0=x) &=& \int_{\sf X} \bP(\Phi_{n_1} \in dy | \Phi_0=x)
\bP(\Phi_{n_1+m}\in B | \Phi_{n_1}=y) \\
&\geq & \int_C \bP(\Phi_{n_1} \in dy | \Phi_0=x) \bP(\Phi_m\in B | \Phi_0=y) \\
&\geq & \delta \phi(B) \bP( \Phi_{n_1} \in C | \Phi_0=x) \;>\;0.
\end{eqnarray*}

We now examine the Markov chain of stage lengths of section \ref{sec:GQ} in this 
general framework.
\begin{proposition} The Markov chain $\boldsymbol{Y}$ of stage lengths with 
state space ${\sf X}=[0,\infty)$ equipped with the Borel $\sigma$-field $\cB$, and 
transition probability kernel $Q(x,A) = \int_A q(x,y)dy$ as defined in Proposition 
\ref{prop:Q} is $\psi$-irreducible and strongly aperiodic.
\end{proposition}
\begin{proof}
The measure $\phi$ defined on $({\sf X},\cB)$ via $\phi(A) = \int_A g(y)dy$ for all
$A \in \cB$ is an irreducibility measure for the Markov chain $\boldsymbol{Y}$. Indeed it is
easy to see that $\phi(A) >0$ implies $L(x,A)>0$. This implies that there exists a maximal
irreducibility measure $\psi$ and that $\boldsymbol{Y}$ is $\psi$-irreducible.
We next show that for any given $a>0$ the set $C_a:=[0,a]$ is a {\em small} set 
for the process $\boldsymbol{Y}$. Define the measure $\phi_a$ on $({\sf X}, \mathcal{B})$ 
via $\phi_a(A) := e^{-\lambda a} \int_A g(y) dy$ for any $A\in \cB$. Note that
\[
Q(x,A) \;=\; \int_A \left( \lambda x e^{-\lambda x \overline{G}(y)}
+ e^{-\lambda x} \right) g(y)dy \;\geq \; e^{- \lambda a} \int_A g(y)dy = 
\phi_a(A),  \hspace{0.3in} \mbox{for all $x \in C$}.
\]
and hence Definition (\ref{def:small}) is satisfied with $\delta = e^{-\lambda a}$, $m=1$. 
Therefore $C_a$ is a small set and furthermore, since $m=1$, the chain $\boldsymbol{Y}$
is strongly aperiodic. Finally the set $C_a$ is accessible and therefore, from Definition 
\ref{def:irreducible} $\boldsymbol{Y}$ is $\psi$-recurrent.  
\end{proof}

The following is Proposition 3.10 of Tweedie \cite[p. 840]{Tweedie}.
\begin{theorem} \label{prop:Tweedie3.10}
Let $P$ be the transition law of a $\psi$-irreducible, aperiodic Markov chain on ${\sf X}$.
The chain is positive recurrent with an invariant probability measure $\pi$ satisfying
$\int_{\sf X} h(x)\pi(dx)< \infty$ for a non-negative function $h$ if and only if there is
a non-negative function $V(x)$, finite $\pi$-almost everywhere, a small set $C$, and
$d>0$ such that 
\begin{equation}  \label{prop3.10}
\int_{\sf X} P(x,dy) V(y) \leq V(x) - h(x) + d \ind_C(x),\hspace{0.3in} x\in {\sf X}.
\end{equation}
\end{theorem}

We now examine specifically the Markov chain $\{Y_n\}$, $n\in \bN$, of consecutive
service stage lengths of Section \ref{sec:GQ} with state space ${\sf X} =\bR^+$ 
equipped with the Borel $\sigma$-field $\cB(\bR^+)$ (which in the
sequel we will abbreviate to $\cB$), and transition
kernel $Q(\cdot,\cdot):{\sf X}\times \cB \rightarrow [0,1]$ given in Proposition \ref{prop:Q}.
Using the above theorem we now proceed to give the proof of Proposition \ref{prop:pi-moments}.
\begin{proof}[Proof of Proposition \ref{prop:pi-moments}]
The state space of the chain is ${\sf X} = [0,\infty)$ and the kernel
$Q: {\sf X} \times \mathcal{B} \rightarrow [0,1]$ is given by (\ref{transition-density}).
We will use Theorem \ref{prop:Tweedie3.10} with $h(x)=e^{R x}$ and 
$V(x) = K e^{R x}$ where $K>1$ and $R$ is the exponent of Assumption \ref{ass:exp}. 
%
%
%
We have 
\begin{equation}  \label{Ix}
\int_0^\infty (\lambda x e^{-\lambda x \overline{G}(y)} + e^{-\lambda x})\, g(y) \, Ke^{R y} dy
= K \int_0^\infty \lambda  x \, e^{-\lambda x \overline{G}(y)}\, g(y) \, e^{R y} dy \;+\;
K \, e^{-\lambda x} \bE [e^{R \sigma_1}]  
\end{equation}
Consider the function $I:[0,\infty) \rightarrow \bR$ defined by 
$I(x):= \int_0^\infty \lambda  e^{-\lambda x \overline{G}(y)} \, g(y) \, e^{R y} dy$. One 
can easily show, using the Dominated Convergence Theorem, that $I$ is a continuous function and 
that $\lim_{x\rightarrow \infty}I(x) =0$. We therefore conclude that $M_1:=\sup_{0\leq x <\infty}I(x)$
is finite. The right hand side of (\ref{Ix}) can be written then as
$I(x) K x +K e^{-\lambda x} \bE[e^{R \sigma_1}]$. We will establish inequality (\ref{prop3.10}) 
with $C:=[0,a]$. Suppose first that $x>a$. Then this inequality is satisfied provided that 
\begin{equation}  \label{Ix1}
M_1 K x +K e^{-\lambda x} \bE[e^{R \sigma_1}] \leq Ke^{R x} - e^{R x} \;\;\;
\mbox{ for all $x > a$}.
\end{equation}
If we choose $a$ sufficiently large so that 
\begin{equation}  \nonumber 
M_1 a + \bE[e^{R \sigma_1}]  \leq e^{R a } \left(1-\frac{1}{K} \right)
\end{equation}
then (\ref{Ix1}) is satisfied. Having chosen $a$, we may chose $d$ so that 
\begin{equation}  \nonumber 
d \geq \frac{K}{K-1} \left( \bE[e^{R \sigma_1}] -1 +M_1 a \right) 
\end{equation}
and then inequality (\ref{prop3.10}) holds for all $x$.
\end{proof}

Therefore the chain has a unique invariant probability measure $\pi$ and we have 
for any Borel set of $\bR^+$ the equation
\[
\pi(B) \;=\; \int_{\bR^+} Q(x,B)  \pi(dx) \;=\;  
\int_{\bR^+} \Big( \int_B   q(x,y)dy \Big) \pi(dx) ;=\;  
\int_B   \Big(  \int_{\bR^+} q(x,y)\pi(dx) \Big) dy .
\]
This shows that $\pi$ is absolutely continuous with respect to the Lebesgue measure
on $\bR^+$ with density $f$ satisfying the relationship (\ref{invariant-density}).

\subsection{Interchange of summation and differentiation using the Cauchy integral formula}
\begin{proposition} \label{prop:Dmphi}
The series (\ref{Generating function}) giving the pgf $\phi(z)$ may be differentiated $m$ times 
inside the summation sign to give (\ref{Generating function2}).
\end{proposition}
\begin{proof} 
With the notation of Theorem \ref{th:MGs1} 
 (\ref{Generating function}) can be written as 
\begin{equation} \label{phi-f-z}
\phi(z)=\sum_{k=1}^\infty (-1)^{k-1} x_k f_k(z).
\end{equation}
Since $\mu>0$ and $B(0)=0$,
\(
0<b_{k+1}\leq b_k\leq b_1<1.
\)
Choose $R\in \left(2,\frac{1}{b_1}\right)$. 
For $|z|\leq R$,
\(
|1-b_kz| \geq 1-b_1R>0,
\)
and hence
\[
|f_k(z)| \leq \frac{R}{1-b_1 R}.
\]
Using this bound, 
\[
\sum_{k=1}^\infty x_k |f_k(z)| \le \frac{R}{1-b_1 R}\sum_{k=1}^\infty x_k<\infty
\]
and therefore, by the Weierstrass $M$-test, the partial sums
\[
S_N(z):=\sum_{k=1}^N(-1)^{k-1}x_k f_k(z)
\]
converge uniformly on $|z|\leq R $ as $N\rightarrow \infty$.

Now choose \(r>0\) sufficiently small so that \( \{z:|z-1|\le r\}\subset\{z:|z|<R\}\).
For every $N$, $S_N$ is holomorphic on this disk. By Cauchy's integral formula for derivatives,
\[
S_N^{(m)}(1) = \frac{m!}{2\pi i} \int_{|\zeta-1|=r} \frac{S_N(\zeta)}{(\zeta-1)^{m+1}}\,d\zeta. 
\]
Since $\phi$ is also holomorphic on this disk,
\[
\phi^{(m)}(1) = \frac{m!}{2\pi i} 
\int_{|\zeta-1|=r} \frac{\phi(\zeta)}{(\zeta-1)^{m+1}}\,d\zeta. 
\]
Subtracting and using the ML--inequality for complex integrals,
\[
\left|S_N^{(m)}(1)-\phi^{(m)}(1)\right|
\; \le \;
\frac{m!}{2\pi} \int_{|\zeta-1|=r} \frac{|S_N(\zeta)-\phi(\zeta)|}{|\zeta-1|^{m+1}}\,|d\zeta| 
\; \le \;
\frac{m!}{r^m} \sup_{|\zeta-1|=r} |S_N(\zeta)-\phi(\zeta)|.
\]
Since $S_N\to\phi$ uniformly on $|z|\leq R$ as $N\rightarrow \infty$, 
the right-hand side tends to 0. Hence
\[
 \phi^{(m)}(1) \;=\; \lim_{N\rightarrow \infty} S_N^{(m)}(1) \;=\; 
 \sum_{k=1}^N (-1)^{k-1}x_k f_k^{(m)}(1).
\]
Taking into account (\ref{Dmfk}) and recalling the definitions (\ref{def_xm}) we obtain (\ref{Generating function2}).
\end{proof}
\paragraph{Acknowledgement:} The authors would like to thank the Associate Editor Prof.\ Jiheng Zhang 
and an anonymous referee for remarks that greatly improved the quality of the paper.
\paragraph{Competing Interests:} The authors declare none.


\begin{thebibliography}{12}

\bibitem{Aldousetal} D.\ Aldous, M.\ Miyazawa and T.\ Rolski (2001). 
On the Stability of a Batch Clearing System with Poisson Arrivals and Subadditive 
Service Times {\em Journal of Applied Probability} {\bf 38}, 3, 621--634.

\bibitem{Alt}
H.\ W.\ Alt (2016). {\em Linear Functional Analysis: An Application-Oriented 
Introduction,} Springer.

\bibitem{Asmussen}  
S.\ Asmussen (2003). \textit{Applied Probability and Queues,} 2nd edition, 
John Wiley.

\bibitem{AH}  B.\ Avi-Itzhak and S.\ Halfin (1989). Response Times in Gated
$M/G/1$ Queues: The Processor-Sharing Case. {\it Queueing Systems} {\bf 4,} 3,
263-279.

\bibitem{Browne1}  S.\ Browne, E.\ G.\ Coffman Jr., E.\ N.\ Gilbert, and P.\ E.\
Wright (1992). Gated, Exhaustive, Parallel Service, {\it Probability in the 
Engineering and the Informational Sciences} {\bf 6}, 217-239.

\bibitem{Browne2}  S.\ Browne, E.\ G.\ Coffman Jr., E.\ N.\ Gilbert, and P.\ E.\
Wright (1992). The Gated Infinite-Server Queue: Uniform Service Times. {\it
SIAM Journal on Applied Mathematics} {\bf 52,} 6, 1751-1762.

\bibitem{Douc} R.\ Douc, E.\ Moulines, P.\ Priouret, P.\ Soulier (2018). {\em
Markov Chains,} Springer.



\bibitem{JS}  D.\ L.\ Jagerman and B.\ Sengupta (1989). A functional equation
arising in a queue with a gating mechanism, {\it Probability in the Engineering 
and the Informational Sciences} {\bf 3}, 417-433.


\bibitem{Kress} 
R.\ Kress (2014). {\em Linear Integral Equations}, 3rd ed.  Springer.

\bibitem{Lebedev1}
A.\ V.\ Lebedev (2003). The Gated Infinite-Server Queue with Unbounded Service 
Times and Heavy Traffic. {\em Problems of Information Transmission} 
{\bf 39}(3), 309-316.

\bibitem{Lebedev2}
A.\ V.\ Lebedev (2004). Gated Infinite-Server Queue with Heavy Traffic and Power 
Tail. {\em Problems of Information Transmission} {\bf 40}(3), 237-242.

\bibitem{Lindner}
M.\ Lindner (2006). {\em Infinite Matrices and Their Finite Sections: An Introduction to 
the Limit Operator Method}, Frontiers in Mathematics, Birkh\"auser, Basel.




\bibitem{MT} S.\ P.\ Meyn and R.\ L.\ Tweedie (2012). {\em Markov Chains and
Stochastic Stability,} 2nd ed. Cambridge University Press.

\bibitem{PM1} D.\ Pinotsi and M.\ A.\ Zazanis (2004). Stability Conditions for 
Gated $M/G/\infty$ Queues, {\em Probability in the Engineering and 
Informational Sciences} {\bf 18}(1) 103-110.

\bibitem{RS} K.\ M.\ Rege and B. Sengupta (1989). A Single Server Queue
with Gated Processor-Sharing Discipline. {\it Queueing Systems} {\bf 4,} 3,
249-261.


\bibitem{Shivakumar-SZ} P.\ N.\ Shivakumar, K.C. Sivakumar, and Yang Zhang (2016).
{\em Infinite Matrices and Their Recent Applications}, Springer, 2016.

\bibitem{Tweedie} R.\ L.\ Tweedie. (2000) Markov Chains: Structure and Applications 
in \emph{Handbook of Statistics} \textbf{19,} ed. D.\ N.\ Shanbhag and C.\ R.\ Rao, 
Elsevier Amsterdam, 817-851.

\bibitem{Knessl1} X.\ Tan and C.\ Knessl (1994). Heavy Traffic Asymptotics for
a Gated, Infinite-Server Queue with Uniform Service Times. {\it SIAM Journal
of Applied Mathematics} {\bf 54}, 6, 1768-1779.

\bibitem{Zazanis} M.\ A.\ Zazanis (2004). Infinite Server Queues with Synchronized 
Departures Driven by a Single Point Process. {\em Queueing Systems} {\bf 48}, 309–338.

\end{thebibliography}
\end{document}